\documentclass[11pt,twoside]{amsart}
\usepackage{amsfonts}
\usepackage{amsmath}
\usepackage{amssymb}
\usepackage[latin1]{inputenc} 
\usepackage{mathrsfs}
\usepackage[dvipsnames]{xcolor}
\theoremstyle{theorem}
\newtheorem{thm}{Theorem}
\newtheorem{pro}{Proposition}
\newtheorem{cor}{Corollary}
\newtheorem{lem}{Lemma}
\newtheorem*{thm A}{Theorem A}
\newtheorem*{thm B}{Theorem B}
\newtheorem*{thm C}{Theorem C}
\newtheorem*{thm D}{Theorem D}

\theoremstyle{definition}
\newtheorem{defn}{Definition}

\theoremstyle{remark}
\newtheorem{rem}{Remark}
\newtheorem{exe}{Example}

\newenvironment{proofof}{{\noindent {\it Proof of }}}{\hfill \finpr \\ }
\newenvironment{pcs}{{\noindent {\bf Special cases. }}}{\hfill \\ }
\numberwithin{equation}{section}

\def\Supp{\mathop{\rm Supp}}

\def\finpr{\hfill \hbox{
\vrule height 1.453ex  width 0.093ex  depth 0ex \vrule height
1.5ex  width 1.3ex  depth -1.407ex\kern-0.1ex \vrule height
1.453ex  width 0.093ex  depth 0ex\kern-1.35ex \vrule height
0.093ex  width 1.3ex  depth 0ex}}
\def\cb{{\Bbb C}}
\def\rb{{\Bbb R}}

          \def\L{\Lambda}         \def\l{\lambda}
\def\nb{\mathbb N}       \def\ds{\displaystyle}  
           \let\L=\longrightarrow  
                  \let\l=\rightarrow
               
                 \def\ds{\displaystyle}

\def\1{1\!\rm l}

\title[ ]{$m$-potential theory and $m$-generalized Lelong numbers associated with $m$-positive supercurrents}
\author[ ]{Fredj Elkhadhra and Khalil Zahmoul}
\address{ University of Sousse\\ Higher School of Sciences and Technology of Hammam Sousse\\ MaPSFA (LR 11 ES 35)\\ 4011 Hammam Sousse\\ Tunisia.}
\email{fredj.elkhadhra@essths.u-sousse.tn, zahmoul.khalil@essths.u-sousse.tn}
\begin{document}
\hyphenpenalty=10000
\begin{abstract} In this study, we first define the local potential associated to a weakly positive closed supercurrent in analogy to the one investigated  by Ben Messaoud and El Mir in the complex setting. Next, we study the definition and the continuity of the $m$-superHessian operator for unbounded $m$-convex functions. As an application, we generalize our previous work on Demailly-Lelong numbers and several related results in the superformalism setting. Furthermore, strongly inspired by the complex Hessian theory, we introduce the Cegrell-type classes as well as a generalization of some $m$-potential results in the class of $m$-convex functions.
\end{abstract}
\maketitle
\section{Introduction} The pluripotential theory in several complex variables, especially the complex Monge-Amp\`ere operator, has generated considerable interest over the past five decades, thanks to the many applications on complex analysis, complex differential geometry and number theory. By letting ourselves inspired by techniques go back to Demailly \cite{3} and Ben Messaoud and El Mir \cite{12}, we extend profound results in pluripotential theory to the superformalism setting via the supercurrent theory originated by Lagerberg \cite{5} and Berndtsson \cite{11}. Our first main result is to associate to a given positive closed supercurrent $T$, a local potential $U$ satisfying many interesting properties similar to the complex setting. Next, we investigate the weak convergence of the superHessian operator relatively to a smooth regularisations of $T$ and $U$ with respect to local uniform converging sequences of convex functions. Comparing with the complex context, this is a clear difference. Namely, it is well known that the analogous weak continuity for the complex Monge-Amp\`ere operator requires a convergence monotonicity condition on the sequences of plurisubharmonic functions. Among other properties of $U$ and as in the paper of Ghiloufi, Zaway and Hawari \cite{0} in the complex setting, we describe a relation between the Lelong numbers of $T$ and of $U$. The second important aim of this paper is the study of the definition and the continuity of the $m$-superHessian operator associated to an $m$-positive closed supercurrent $T$, for some classes of unbounded $m$-convex functions as well as when we consider a regularization sequence of $T$. Recall that we consider the concepts of $m$-positivity of superforms and supercurrents introduced by Elkhadhra and Zahmoul \cite{18} analogously with the complex Hessian theory. As an immediate consequence we introduce the $m$-generalized Lelong number of $T$ relatively to a given $m$-convex function. It is worth mentioning that the results obtained here are the counterparts of corresponding one in the complex Hessian pluripotential theory and they had a crucial role in the development of complex geometry. We hope that our superformalism adaptation attracts some problems within the tropical geometry and the Riemannian geometry. Next, as a generalization of the work of Labutin \cite{26} and the work of Trudinger and Wang \cite{15}, we obtained new outcomes on the weighted $m$-Hessian capacity for $m$-convex functions on a Borel subset of $\rb^n$, that are real counterparts of recent results due to Nguyen \cite{25} on generalized $m$-capacity for $m$-subharmonic functions. Let us describe more precisely the content of the paper. Besides the introduction, this paper has four sections. In Section 2, we recall some basic notations and definitions necessarily for the rest of this paper. The third section is where we introduce and investigate a variant of important related properties of the local potential $U$ of a positive closed current $T$ in the superformalism setting. Also, we investigate the weak convergence of the sequence of the superHessian operators $(U_j\wedge dd^{\#}u_1^j\wedge...\wedge dd^{\#}u_k^j)_j$, where $U_j$ is a smooth regularization by convolution of $U$ and $(u_s^j)_j$ are sequences of locally uniformly converging convex functions. The same weak convergence for $T_j$ instead of $U_j$ was guaranted by means of the link between $T$ and its potential. Moreover, we establish a result concerning the relation between the Lelong number of $T$ and of $U$. Let $\Omega$ be an open bounded subset of $\rb^n$ and denote by $\beta=\frac{1}{2}dd^{\#}|x|^2$ the K\"ahler form on $\rb^n\times\rb^n$. Strongly motivated by the work of Demailly \cite{3} on the Monge-Amp\`ere operator, we study in Section 4 the definition and the continuity of the $m$-superHessian operator $T\wedge\beta^{n-m}\wedge dd^{\#}u_1\wedge...\wedge dd^{\#}u_k$ for monotone decreasing sequences of $m$-convex functions bounded near $\partial\Omega\cap\Supp T$. Here $T$ is a closed $m$-positive supercurrent of bidimension $(p,p)$ on $\Omega\times\rb^n$ such that $T\wedge\beta^{n-m}$ is positive and $\Supp T$ is the support of $T$.  As an application, we introduce the $m$-generalized Lelong number of $T$ relatively to an $m$-convex weight $\varphi$ as the measure of the asymptotic behaviour of the supercurrent $T\wedge\beta^{n-m}\wedge(dd^{\#}\varphi)^{m+p-n}$ near the $m$-polar set $\{\varphi=-\infty\}$. In analogy with the complex case, the remaining of Section 4 contains several topics in superformalism setting, including the counterpart of the famous Demailly comparison theorem and the transformation of the Lelong number by a direct image of a projection. In the last section, we introduce two classes of $m$-convex functions $\mathcal{E}_{m}^0(\Omega)$ and $\mathcal{F}_m(\Omega)$. These classes turn out to the Cegrell classes of plurisubharmonic functions which are a basic tool in the study of interesting problems in the complex setting. Namely, the domain of definition of the Monge-Amp\`ere operator. After establishing some elmentary properties of such classes in the superformalism setting and in analogy with Nguyen \cite{25}, we introduce and investigate the weighted relative extremal function $R_{m,u}(E)$ as a generalization of the one given by Labutin \cite{26}. We also define the weighted $m$-Hessian capacity $cap_{m,u}(E)$ for an $m$-convex weight $u$. Finally, we discover a link between $cap_{m,u}(E)$ and $R_{m,u}(E)$  for Borel compact subset $E$ as well as for $u\in\mathcal{F}_m(\Omega)$.
\section{Preliminaries} In this part, we recall definitions and basic properties of the superforms and supercurrents as introduced by Lagerberg \cite{5}, Berndtsson \cite{11} and  Elkhadhra and Zahmoul \cite{18}, that will be used throughout this paper. Let $V$ and $W$ to be two $n$-dimensional vector spaces over $\rb$, so that $x=(x_{1},...,x_{n})$ and $\xi=(\xi_{1},...,\xi_{n})$ are the corresponding coordinates, and let $J:V \rightarrow W$ be an isomorphism such that $J(x)=\xi$, and denote its inverse by $J(\xi)=-x$ if $x\in V$ is the element for which $J(x)=\xi$. By setting $ E=V \times W=\lbrace(x,\xi);\,x\in V,\,\xi\in W\rbrace$, we observe that the map $J$ can be extended over $E$ by setting $J(x,\xi)=(J(\xi),J(x))$, so that $J^2=-id$. For $0\leqslant p,q\leqslant n$, any smooth superform on $E$ of bidegree $(p,q)$ can be written as
$$\alpha=\sum_{K,L}\alpha_{KL}(x)dx_{K}\wedge d\xi_{L},$$
where $K=(k_{1},...,k_{p})$, $dx_{K}=dx_{k_{1}}\wedge ...\wedge dx_{k_{p}}$, $L=(l_{1},...,l_{q})$, $d\xi_{L}=d\xi_{l_{1}}\wedge ...\wedge d\xi_{l_{q}}$ and each coefficient $\alpha_{KL}(x)$ is smooth and depends only on $x$. We denote by ${\mathscr E}^{p,q}:={\mathscr E}^{p,q}(E)$ the set of smooth superforms on $E$ of bidegree $(p,q)$. In particular, if $p=q$ we say that $\alpha$ is symmetric if and only if $\alpha_{KL}=\alpha_{LK}$, $\forall K,L$. It is clear that $J^{\ast}(dx_{i})=d\xi_{i}$ and ${J}^{\ast}(d\xi_{i})=-dx_{i}$. In order to simplify the notation, we denote the operator $J^\ast$ by $J$, which can be extended on ${\mathscr E}^{p,q}$  by a map denoted again by $J :{\mathscr E}^{p,q}\rightarrow{\mathscr E}^{q,p}$ and defined by
$$J(\alpha)=J\left(\sum_{K,L} \alpha_{KL}(x)dx_{K}\wedge d\xi_{L}\right)=(-1)^q\sum_{K,L} \alpha_{KL}(x)d\xi_{K}\wedge dx_{L},\quad \forall\alpha\in{\mathscr E}^{p,q}.$$
In particular, if $\alpha\in{\mathscr E}^{p,p}$, then $\alpha$ is symmetric if and only if  $J(\alpha)=\alpha$. Now, let us consider the K\"ahler form on $\rb^n\times\rb^n$ to be $$\beta:=\sum_{i=1}^{n}dx_{i}\wedge d\xi_{i}\in{\mathscr E}^{1,1}.$$ A simple computation gives $\beta^{n}=n!dx_{1}\wedge d\xi_{1}\wedge...\wedge dx_{n}\wedge d\xi_{n}$. Next, we introduce three notions of positivity on ${\mathscr E}^{p,p}$. A superform $\varphi\in{\mathscr E}^{n,n}$ is said to be positive ($\varphi\geqslant 0$) if $\varphi=f\beta^{n}$, where $f$ is a positive function. Let $\varphi\in{\mathscr E}^{p,p}$ be  symmetric, we say that $\varphi$ is:
  \begin{enumerate}
 \rm\item  weakly positive if $\varphi\wedge \alpha_{1}\wedge J(\alpha_{1})\wedge ...\wedge \alpha_{n-p}\wedge J(\alpha_{n-p})\geqslant 0,\ \forall\alpha_{1},...,\alpha_{n-p}\in{\mathscr E}^{1,0}.$
 \rm\item  positive if $\varphi\wedge \sigma_{n-p}\alpha\wedge J(\alpha)\geqslant 0,\ \forall\alpha\in{\mathscr E}^{n-p,0};\ \sigma_{k}=(-1)^{\frac{k(k-1)}{2}}.$
 \rm\item  strongly positive if $\varphi=\sum_{s=1}^{N}\lambda_{s}\alpha_{1,s}\wedge J(\alpha_{1,s})\wedge...\wedge \alpha_{p,s}\wedge J(\alpha_{p,s});\ \lambda_{s}\geqslant 0,\  \alpha_{i,s}\in{\mathscr E}^{1,0}.$
  \end{enumerate}
For every $\alpha\in{\mathscr E}^{n,n}$, there exists a function $\alpha_0$ defined on $V$ such that  $$\alpha=\alpha_{0}dx_{1}\wedge d\xi_{1}\wedge...\wedge dx_{n}\wedge d\xi_{n}.$$  When an orientation on $V$ is chosen and $\alpha_0$ is integrable, the superintegral of $\alpha$ is defined by  $$\int_{E}\alpha=\int_{V}\alpha_{0}dx_{1}\wedge ...\wedge dx_{n}.$$
The operators $d$ and $d^{\#}$ are of type $(1,0)$ and $(0,1)$ respectively and acting on ${\mathscr E}^{p,q}$ by the following expressions $$d=\sum_{i=1}^{n}\partial_{x_{i}}dx_{i}\qquad {\rm and}\qquad d^{\#}=\sum_{j=1}^{n}\partial_{x_{j}}d\xi_{j}.$$ It is easy to see that $d^{2}=(d^{\#})^{2}=0$, $dd^{\#}=-d^{\#}d$ and $\beta=\frac{1}{2}dd^{\#}|x|^2$. Moreover, in this situation we can present a Stokes' formula as follow: Assume that $\Omega\subset V$ is an open bounded set with smooth boundary and let $\alpha\in{\mathscr E}^{n-1,n}$. Then, $$\int_{\Omega\times W}d\alpha=\int_{\partial\Omega\times W}\alpha.$$
Denote by ${\mathscr D}^{p,q}$ the set of smooth and compactly supported superforms of bidegree $(p,q)$ on $E$, whose topology can be defined by means of the inductive limit. We introduce the space ${\mathscr D}_{p,q}$ of supercurrents of  bidimension $(n-p,n-q)$ as the topological dual of ${\mathscr D}^{n-p,n-q}$. This means that a supercurrent $T$ of bidimension $(n-p,n-q)$ is nothing but a continuous linear map on ${\mathscr D}^{n-p,n-q}$. More precisely, $T$ is a superform of bidegree $(p,q)$ which has distributions coefficients depending only on $x$. That is   $$T=\sum_{| I|=p,| J|=q}T_{IJ}dx_{I}\wedge d\xi_{J},$$ where $T_{IJ}$ are distributions defined uniquely on $x$. In particular, as with superforms if $p=q$ we say that the supercurrent $T$ is symmetric if and only if $T_{IJ}=T_{JI}$  $\forall I,J$.  For any $\alpha\in {\mathscr D}^{n-p,n-q}$, denote by $\langle T,\alpha\rangle$ the action of $T$ on $\alpha$. A supercurrent $T$ is said to be closed if $dT=0$.
Assume that $T$ is symmetric and of bidegree $(p,p)$. $T$ is said to be:
\begin{enumerate}
\item weakly positive if $\langle T,\alpha\rangle\geqslant 0$ for any $\alpha\in {\mathscr D}^{n-p,n-p}$ strongly positive.
\item positive if $\langle T,\sigma_{n-p}\alpha\wedge J(\alpha)\rangle\geqslant 0$ for any $\alpha\in {\mathscr D}^{n-p,0}$.
\item strongly positive if  $\langle T,\alpha\rangle\geqslant 0$ for any $\alpha\in {\mathscr D}^{n-p,n-p}$ weakly positive.
\item convex if  $dd^{\#}T$ is weakly positive and concave if $-T$ is convex.
\end{enumerate}
For $K\Subset\rb^n$ and $T$ a supercurrent of order zero, we define the mass measure of $T$ on $K$ by $$\|T\|_K=\sum_{IJ}|T_{IJ}|(K),$$ where $T_{IJ}$ are the coefficients of $T$. Due to Lagerberg \cite{5}, the mass of a positive supercurrent $T$ of bidimension $(p,p)$ on $K$ is proportional to the positive measure $T\wedge\beta^p(K)$. Now, we adapt to the superformalism context the following notions of $m$-positivity and $m$-convexity from Elkhadhra and Zahmoul \cite{18} in the complex hessian setting:
\begin{enumerate}
\item A symmetric superform $\alpha$ of bidegree $(1,1)$ is said $m$-positive if at every point we have
$\alpha^j\wedge\beta^{n-j}\geqslant 0,\ \forall j=1,...,m$.
\item A symmetric supercurrent $T$ of bidimension $(n-p,n-p)$ such that $p\leqslant m\leqslant n$, is called $m$-positive if we have $T\wedge\beta^{n-m}\wedge\alpha_1\wedge...\wedge\alpha_{m-p}\geqslant 0,$ for all $m$-positive superforms $\alpha_1,...,\alpha_{m-p}$ of bidegree $(1,1)$.
\item A function $u: V\L\rb\cup\{-\infty\}$ is called $m$-convex if it is subharmonic and the supercurrent $dd^{\#}u$ is $m$-positive. Denote by ${\mathscr C}_m$ the set of $m$-convex functions.
\end{enumerate}
Next, let us recall some basic facts about $m$-convex functions due to Wan and Wang \cite{17}:
\begin{enumerate}
\item If $u$ is of class ${\mathscr C}^2$ then $u$ is $m$-convex if and only if $dd^{\#}u$ is $m$-positive superform.
\item convex functions=${\mathscr C}_n\subset{\mathscr C}_{n-1}\subset\cdots\subset{\mathscr C}_1$= subharmonic functions.
\item If $u$ is $m$-convex then the standard regularization $u_j=u\ast\chi_j$ is smooth and $m$-convex. Moreover, $(u_j)_j$ decreases pointwise to $u$.
\item Let $u,v\in{\mathscr C}_m$ then $\max(u,v)\in{\mathscr C}_m$.
\item If $(u_\alpha)_\alpha\subset{\mathscr C}_m$, $u=\sup_\alpha u_\alpha<+\infty$ and $u$ is upper semi-continuous then $u$ is $m$-convex.
\end{enumerate}
For the sake of simplicity, in the rest of this paper we consider two copies of $\rb^n$, i.e. $V=W=\rb^n$ and we say form instead of superform and current instead of supercurrent.
\section{Local potential associated with positive closed supercurrent}
\subsection{Definition and properties} In complex analysis, a possible way of studying interesting properties of a given current is to consider the associated local potential. Namely, it was stressed by Ben Messaoud and El Mir \cite{12} that the local potential associated to a positive closed current in $\cb^n$ is crucial in the study of the complex Monge-Amp\`ere operator. By a strong analogy with the complex theory, we try to extend this notion to the superformalism setting as follows:
\begin{defn}\label{d1} Assume that $T$ is a weakly positive closed current of measure coefficients and of bidimension $(p,p)$ on $\rb^n\times\rb^n$ with $1\leqslant p\leqslant n-1$. Let $\Omega\Subset\rb^n$ and $\eta$ be a smooth compactly supported function on $\rb^n$ such that $0\leqslant\eta\leqslant1$ and $\eta\equiv1$ on $\overline{\Omega}$. The local  potential $U=U(\eta,T)=U(\Omega,T)$ associated to $T$ is a weakly negative current of bidimension $(p+1,p+1)$ on $\rb^n\times\rb^n$ defined by
$$U(x,\xi)=-c_n\int_{(y,\zeta)\in\rb^n\times\rb^n}\eta(y)T(y,\zeta)\wedge\frac{\beta^{n-1}(x-y,\xi-\zeta)}{|x-y|^{n-2}},$$
where $c_n=\frac{1}{(n-2)Vol_n(B(0,1))}$.
\end{defn}
Analogously to the complex setting, if $h(x)=-c_n|x|^{2-n}$, it is important to observe that the coefficients of $U$ are nothing but the convolutions of the measure coefficients of $\eta T$ by the function $h$ which is in $L_{loc}^{1+\frac{1}{n}}$. Therefore, the coefficients of $U$ are in $L_{loc}^{1+\frac{1}{n}}$. Moreover, we can easily deduce the positivity of the local potential by using the equality (\ref{0}) in the proof of the following Lemma \ref{lmm1} and the fact that the push forward of a (weakly) positive current is (weakly) positive thanks to Lagerberg \cite{5}.
\begin{exe}\label{x1}\
\begin{enumerate}
\item Assume that $T$ is a weakly positive closed current of bidimension $(n-1,n-1)$. Thanks to Lagerberg \cite{5}, there exists a convex function $u$ such that $T=dd^{\#}u$. A direct computation gives
$$\begin{array}{lcl}
U(x,\xi)&=&-c_n\ds\int_{(y,\zeta)\in\rb^n\times\rb^n}\eta(y)dd^{\#}u(y,\zeta)\wedge\frac{\beta^{n-1}(y,\zeta)}{|x-y|^{n-2}}\\&=&-c_n\ds\int_{(y,\zeta)\in\rb^n\times\rb^n}\frac{\eta(y)\Delta u(y)\beta^n(y,\zeta)}{|x-y|^{n-2}}.
\end{array}$$
It follows that for $x\in\Omega$, we have $U(x,\xi)=(\eta\Delta u)\ast h(x)=\Delta(\eta u)\ast h+v(x),$ where $v$ is smooth because it is a product of convolution of $h$ with measures involving derivations of $\eta$ which vanish on $\Omega$. Since $\Delta h=\delta_0$, we get $U(x,\xi)=\eta(x)u(x)+ v(x)=u(x)+v(x)$ on $\Omega$. \\
\item Let $M$ be a minimal, smooth and $p$-dimensional submanifold of $\rb^n$, and let $[M]_s$ be the associated minimal current as introduced by Berndtsson \cite{11}. This means that $[M]_s\wedge\beta^{p-1}$ is a weakly positive closed current of bidimension $(1,1)$ on $\rb^n\times\rb^n$. The local potential associated to the current $[M]_s\wedge\beta^{p-1}$ is
$$\begin{array}{lcl}
U(x,\xi)&=&-c_n\ds\int_{(y,\zeta)\in\rb^n\times\rb^n}\eta(y)[M]_s\wedge\beta^{p-1}(y,\zeta)\wedge\frac{\beta^{n-1}(x-y,\xi-\zeta)}{|x-y|^{n-2}}\\&=&\ds-p!(n-1)c_n\left(\int_{(y,\zeta)\in\rb^n\times\rb^n}\frac{\eta(y)}{|x-y|^{n-2}}\sigma_M\right)\beta^{n-2}(x,\xi)\\&-&\ds(n-1)(n-2)c_n\int_{(y,\zeta)\in\rb^n\times\rb^n}\eta(y)[M]_s\wedge\beta^{p-1}(y,\zeta)\wedge\beta(x,\zeta)\wedge\beta(y,\xi)\wedge\frac{\beta^{n-3}(x,\xi)}{|x-y|^{n-2}},
\end{array}$$
where $\sigma_M:=[M]_s\wedge\frac{\beta^{p}}{p!}$ is the volume form defined by Berndtsson \cite{11}.
\end{enumerate}
\end{exe}
Let $\chi$ be a radial smooth function and with support included in the unit ball such that $\int\chi(x)dx=1$, and let $\chi_{j}(x)=j^n\chi(jx)$ for $j\in\nb^\ast$. Let us set
$$U_j(x,\xi)=\int_{(y,\zeta)\in\rb^n\times\rb^n}\eta(y)h\ast\chi_j(x-y)T(y,\zeta)\wedge\beta^{n-1}(x-y,\xi-\zeta).$$
Since $h$ is subharmonic on $\rb^n$, the sequence $(h\ast\chi_j)_j$ is smooth and monotone decreasing. It follows that $(U_j)_j$ is a sequence of forms of bidegree $(n-p-1,n-p-1)$ on $\rb^n\times\rb^n$ which decreases weakly to $U$. Assume that
$$U(x,\xi)=\sum_{|I|=|J|=n-p-1}U_{IJ}(x)dx_J\wedge d\xi_J$$
is the canonical expression of $U$. Then, it is not hard to see that $$U_{II}(x)=U\wedge(-1)^{(p+1)(n-p-1)}dx_{I^c}\wedge d\xi_{I^c},\ \mathrm{where}\ I^c=\{1,...,n\}\smallsetminus I.$$ Hence, in view of the above observation, $U_{II}$ is a subharmonic function on $\rb^n$, since it coincides with the convolution of a positive measure by the function $h$. Moreover, we have
\begin{pro}\label{pro1} Let $u=\sum_{|I|=n-p-1}U_{II}$. Then, for all $x\in\rb^n$,
$$u(x)=\frac{(n-p)(n-1)!}{p!}\int_{\rb^n\times\rb^n}\eta(y)h(x-y)T(y,\zeta)\wedge\beta^p(y,\zeta).$$
\end{pro}
\begin{proof} Since the desired equality is made up of two subharmonic functions, it suffices to show it almost everywhere. Clearly we have $\beta(x-y,\xi-\zeta)=\beta(x,\xi)-\beta(x,\zeta)-\beta(y,\xi)+\beta(y,\zeta)$. By the expression of the local potential, observe that the form $\beta^{n-1}(x-y,\xi-\zeta)$ contribute only by its $(p,p)$-component in $(y,\zeta)$ and $(n-p-1,n-p-1)$-component in $(x,\xi)$, which is equals to
$$\sum_{s=0}^{N}\frac{(n-1)!}{(n-p-1-s)!(p-s)!(s!)^2}\beta^{p-s}(y,\zeta)\wedge(\beta(x,\zeta)\wedge\beta(y,\xi))^s\wedge\beta^{n-p-1-s}(x,\xi),$$
for $N=\min(p,n-p-1)$. Since
$$\begin{array}{lcl}
(\beta(x,\zeta)\wedge\beta(y,\xi))^s&=&\ds\left(\left(\sum_{i=1}^{n}dx_i\wedge d\zeta_i\right)\wedge\left(\sum_{j=1}^{n}dy_j\wedge d\xi_j\right)\right)^s\\&=&\ds\left(\sum_{i,j=1}^{n}dx_i\wedge d\zeta_i\wedge dy_j\wedge d\xi_j\right)^s\\&=&{(-1)^s}(s!)^2\ds\sum_{|I|=|J|=s}dx_I\wedge d\xi_J\wedge dy_J\wedge d\zeta_I.
\end{array}$$
Then,
\begin{equation}\label{e10}
U(x,\xi)=\sum_{s=0}^{N}\frac{(n-1)!{(-1)^s} }{(n-p-1-s)!(p-s)!}\left(\sum_{|I|=|J|=s}B_{IJ}(x)\beta^{n-p-1-s}(x,\xi)\wedge dx_I\wedge d\xi_J\right),
\end{equation}
where $B_{IJ}(x)=\int_{\rb^n\times\rb^n}\eta(y)h(x-y)T(y,\zeta)\wedge\beta^{p-s}(y,\zeta)\wedge dy_J\wedge d\zeta_I$. Consequently, we get
$$U\wedge\beta^{p+1}(x,\xi)=\sum_{s=0}^{N}\frac{(n-1)!{(-1)^s}}{(n-p-1-s)!(p-s)!}\left(\sum_{|I|=s}B_{II}(x)\beta^{n-s}(x,\xi)\wedge dx_I\wedge d\xi_I\right).$$
Since $\beta^{n-s}(x,\xi)\wedge dx_I\wedge d\xi_I=\frac{(n-s)!{\sigma_s}}{n!}\beta^{n}(x,\xi)$ and ${\sigma_s}\sum_{|I|=s}dy_I\wedge d\zeta_I=\frac{1}{s!}\beta^{s}(y,\zeta)$, then
$$U\wedge\beta^{p+1}(x,\xi)=\sum_{s=0}^{N}\frac{(n-1)!(n-s)!{(-1)^s}}{(n-p-1-s)!(p-s)!s!n!}\left(\int_{\rb^n\times\rb^n}\eta(y)h(x-y)T(y,\zeta)\wedge\beta^{p}(y,\zeta)\right)\beta^{n}(x,\xi).$$
To calculate $c(N)=\sum_{s=0}^{N}\frac{(n-1)!(n-s)!{(-1)^s}}{(n-p-1-s)!(p-s)!s!n!}$ and since the above argument is still holds when $T$ is only of order zero, let us take $T=\delta_0\beta^{n-p}$. According to the previous equality, one has  $U\wedge\beta^{p+1}(x,\xi)=n!c(N)\eta(0)h(x)\beta^{n}(x,\xi)$. Moreover, by the integral expression of $U$ and by changing variables $((y,\zeta),(x,\xi))\rightarrow((t,\mu),(x,\xi))=((x-y,\xi-\zeta),(x,\xi))$, we obtain
$$U(x,\xi)=\int_{\rb^n\times\rb^n}\eta(x-t)h(t)\delta_x(t)\beta^{n-p}(x-t,\xi-\mu)\wedge\beta^{n-1}(t,\mu).$$
We need only the $(1,1)$-component in $(t,\mu)$ and the $(n-p-1,n-p-1)$-component in $(x,\xi)$ of the form $\beta^{n-p}(x-t,\xi-\mu)$, which is equals to
$$w(x,\xi,t,\mu)=(n-p)\beta(t,\mu)\wedge\beta^{n-p-1}(x,\xi)-(n-p-1)(n-p)\beta^{n-p-2}(x,\xi)\wedge\sum_{i,j=1}^{n}dx_i\wedge d\xi_j\wedge dt_j\wedge d\mu_i.$$
Hence,
$$\begin{array}{lcl}
& &w(x,\xi,t,\mu)\wedge\beta^{n-1}(t,\mu)=\\&=&\ds(n-p)\beta^n(t,\mu)\wedge\beta^{n-p-1}(x,\xi)-\frac{(n-p)(n-p-1)}{n}\beta^n(t,\mu)\wedge\beta^{n-p-1}(x,\xi)\\&=&\ds\frac{(n-p)(p+1)}{n}\beta^{n-p-1}(x,\xi)\wedge\beta^n(t,\mu)
\end{array}$$
and $U\wedge\beta^{p+1}(x,\xi)=\frac{(n-p)(p+1)}{n}n!\eta(0)h(x)\beta^{n}(x,\xi)$. It follows that $c(N)=\frac{(n-p)(p+1)}{n}$. Then, by going back to the canonical expression of $U$, we deduce that
$$\begin{array}{lcl}
u(x)\beta^n(x,\xi)&=&\ds\frac{n!}{(p+1)!}U(x,\xi)\wedge\beta^{p+1}(x,\xi)\\&=&\ds\frac{(n-p)(n-1)!}{p!}\left(\int_{\rb^n\times\rb^n}\eta(y)h(x-y)T(y,\zeta)\wedge\beta^p(y,\zeta)\right)\beta^n(x,\xi).
\end{array}$$
\end{proof}
\begin{pro}\label{ppt} Let $T$ be a positive closed current of bidimension $(p,p)$ on $\rb^n\times\rb^n$. If $x_0\in\rb^n$ and $K$ is a compact of $\rb^n$, then:\
\begin{enumerate}
\item There exists a constant $c=c(\eta,x_0,n,p)>0$ such that $\|\eta T\|_{\rb^n}\leqslant-cu(x_0)$.
\item There exists a constant $c_K=c(K,n,p)\geqslant0$ such that $\|U\|_K\leqslant c_K\|\eta T\|_{\rb^n}$.
\end{enumerate}
\end{pro}
\begin{proof} (1) There is nothing to prove when $u(x_0)=-\infty$. The function $g:x\mapsto-h(x_0-x)$ is lower semi-continuous, then $g$ reaches its minimum on the compact subset $\Supp \eta$. By Proposition 4.1 in \cite{5}, there exists $c_1>0$ such that $$\|\eta T\|_{\rb^n}\leqslant c_1\int_{\rb^n\times\rb^n}\eta(y)T(y,\zeta)\wedge\beta^p(y,\zeta).$$ Then, by Proposition \ref{pro1}, it suffices to take $c=\frac{c_1p!}{(n-p)(n-1)!\ds\min_{x\in\Supp \eta}g(x)}$.

(2) Since $U$ is a negative current of bidimension $(p+1,p+1)$ on $\rb^n\times\rb^n$. Then, by Proposition 4.1 in \cite{5}, there exists a constant $c=c(n,p)\geqslant0$ such that $$\|U\|_K\leqslant-c\int_{K\times\rb^n}U(x,\xi)\wedge\beta^{p+1}(x,\xi).$$
By the integral formula of $U$ and Fubini's theorem, the previous inequality becomes $$\|U\|_K\leqslant \frac{c}{n!}\int_{(y,\zeta)\in\rb^n\times\rb^n}\eta(y)T(y,\zeta)\wedge\beta^{p}(y,\zeta)\left(\int_{x\in K}\frac{d\lambda}{|x-y|^{n-2}}\right),$$ where $d\lambda=dx_1\wedge...\wedge dx_n$. Let $r>0$ such that $K\subset B(0,r)$. The invariance by the orthogonal group $O(\rb^n)$ of the Lebesgue's measure $d\sigma=\sum_{i=1}^{n}(-1)^{i-1}dx_1\wedge...\wedge dx_{i-1}\wedge dx_{i+1}\wedge...\wedge dx_n$ on the unit sphere of $\rb^n$, gives $$\int_{S(0,1)}\frac{d\sigma(x)}{|x-y|^{n-2}}=\min\left(1,\frac{1}{|y|^{n-2}}\right)\leqslant1.$$
Which conclude our proof.
\end{proof}
Following \cite{5}, the convolution of a given current $T(x,\xi)=\sum_{IJ}T_{IJ}(x)dx_I\wedge d\xi_J$ on $\rb^n\times\rb^n$ with $\chi_j$ is defined by $$(T\ast\chi_j)(x,\xi)=\sum_{IJ}(T_{IJ}\ast\chi_j)(x)dx_I\wedge d\xi_J,$$ and $(T\ast\chi_j)_j$ converges weakly to $T$, as $j\l +\infty$. The following lemma is fundamental for what is to come, indeed it gives the relationship between the current $T$ and its local potential $U$. Moreover, it is the counterpart in the superformalism setting of Lemma 5 in \cite{12}.
\begin{lem}\label{lmm1}\
\begin{enumerate}
\item Let $K(x,\xi)=\frac{h(x)}{n!}\beta^{n-1}(x,\xi)$ and $\beta_n(x,\xi)=\frac{1}{n!}\beta^n(x,\xi)=dx_1\wedge d\xi_1\wedge...\wedge dx_n\wedge d\xi_n$. Then, we have $dd^{\#}K(x,\xi)=\delta_0\beta_n(x,\xi)$.
\item $dd^{\#}K_j(x,\xi)=\chi_j(x)\beta_n(x,\xi)$, where $K_j(x,\xi)=(K\ast\chi_j)(x,\xi)$.
\item $U\ast\chi_j=U_j$.
\item There is a smooth form $R_j$ of bidegree $(n-p,n-p)$ on $\rb^n\times\rb^n$, such that $dd^{\#}U_j=(\eta T)\ast\chi_j+R_j$. Moreover, the sequence $(R_j)_j$ converges in $\mathscr C^{\infty}_{n-p,n-p}$ to a form $R$ of bidegree $(n-p,n-p)$ satisfying the equality $dd^{\#}U=T+R$ on $\Omega\times\rb^n$.
\end{enumerate}
\end{lem}
\begin{proof} (1) For all $(x,\xi)\in\rb^n\times\rb^n$, we have
$$\quad dd^{\#}K(x,\xi)=\frac{1}{n!}dd^{\#}h(x)\wedge\beta^{n-1}(x,\xi)=\frac{1}{n!}\Delta(h(x))\beta^n(x,\xi)=\frac{1}{n!}\delta_0\beta^n(x,\xi)=\delta_0\beta_n(x,\xi).$$

(2) For all $(x,\xi)\in\rb^n\times\rb^n$, we have
$$dd^{\#}K_j(x,\xi)=dd^{\#}(K\ast\chi_j)(x,\xi)=(dd^{\#}K\ast\chi_j)(x,\xi)=(\delta_0\ast\chi_j)(x)\beta_n(x,\xi)=\chi_j(x)\beta_n(x,\xi).$$

(3) The equation (\ref{e10}), implies that
$$(U\ast\chi_j)(x,\xi)=\sum_{s=0}^{N}\frac{(n-1)!{(-1)^s}}{(n-p-1-s)!(p-s)!}\left(\sum_{|I|=|J|=s}(B_{IJ}\ast\chi_j)(x)\beta^{n-p-1-s}(x,\xi)\wedge dx_I\wedge d\xi_J\right).$$
Now, since the coefficients of $T$ are measures and depending only on $y$, it is clear that if we put $S_{IJ}(y,\zeta)=\eta(y)T(y,\zeta)\wedge\beta^{p-s}(y,\zeta)\wedge dy_J\wedge d\zeta_I$, then we get
$$\begin{array}{lcl}
(B_{IJ}\ast\chi_j)(x)&=&((S_{IJ}\ast h)\ast\chi_j)(x,\xi)=(S_{IJ}\ast(h\ast\chi_j))(x,\xi)\\&=&\ds\int_{\rb^n\times\rb^n}\eta(y)(h\ast\chi_j)(x-y)T(y,\zeta)\wedge\beta^{p-s}(y,\zeta)\wedge dy_J\wedge d\zeta_I,
\end{array}$$
which completes the proof of the statement.

(4) We denote by $p_1$ and $p_2$ the projections of $(\rb^n\times\rb^n)\times(\rb^n\times\rb^n)$ on $\rb^n\times\rb^n$ such that $p_1((x,\xi),(y,\zeta))=(x,\xi)$ and $p_2((x,\xi),(y,\zeta))=(y,\zeta)$. Let $\tau((x,\xi),(y,\zeta))=(x-y,\xi-\zeta)$, when we integrate on the fibers of $p_2$, we get
\begin{equation}\label{0}
U={p_{2}}_\ast[p_{1}^{\ast}(\eta T)\wedge\tau^\ast(K)]\qquad \mathrm{and}\qquad U_j={p_{2}}_\ast[p_{1}^{\ast}(\eta T)\wedge\tau^\ast(K_j)].
\end{equation}
Since $p_{2\ast}$ and $p_1^\ast$ are commute with $d$ and $d^{\#}$ (see \cite{5}), we obtain $$dd^{\#}U_j={p_{2}}_\ast[p_{1}^{\ast}(\eta T)\wedge\tau^\ast(\chi_j{\beta_n})]+R_j\qquad \mathrm{and}\qquad dd^{\#}U={p_{2}}_\ast[p_{1}^{\ast}(\eta T)\wedge\tau^\ast(\delta_0{\beta_n})]+R,$$ where $$R_j={p_{2}}_\ast[p_{1}^{\ast}(dd^{\#}\eta\wedge T)\wedge\tau^\ast(K_j)]-{p_{2}}_\ast[p_{1}^{\ast}(d^{\#}\eta\wedge T)\wedge\tau^\ast(dK_j)]+{p_{2}}_\ast[p_{1}^{\ast}(d\eta\wedge T)\wedge\tau^\ast(d^{\#}K_j)]$$ and $$R={p_{2}}_\ast[p_{1}^{\ast}(dd^{\#}\eta\wedge T)\wedge\tau^\ast(K)]-{p_{2}}_\ast[p_{1}^{\ast}(d^{\#}\eta\wedge T)\wedge\tau^\ast(dK)]+{p_{2}}_\ast[p_{1}^{\ast}(d\eta\wedge T)\wedge\tau^\ast(d^{\#}K)]$$ are two forms of bidegree $(n-p,n-p)$ on $\rb^n\times\rb^n$. By writing each term of $R$ in the integral form as in Definition \ref{d1} and using the fact that $\eta\equiv1$ on $\overline\Omega$, it is obvious that $R$ is smooth on $\Omega\times\rb^n$. If $T(y,\zeta)=\sum_{|I|=|J|=n-p}T_{IJ}(y)dy_I\wedge d\zeta_J$, $f_j(x,\xi)={p_{2}}_\ast[p_{1}^{\ast}(\eta T)\wedge\tau^\ast(\chi_j{\beta_n})]$ and $g(x,\xi)={p_{2}}_\ast[p_{1}^{\ast}(\eta T)\wedge\tau^\ast(\delta_0{\beta_n})]$, then
$$f_j(x,\xi)=\frac{1}{n!}\sum_{|I|=|J|=n-p}\int_{(y,\zeta)\in\rb^n\times\rb^n}\eta(y)\chi_j(x-y)T_{IJ}(y)dy_I\wedge d\zeta_J\wedge\beta^n(x-y,\xi-\zeta)$$ and
$$g(x,\xi)=\frac{1}{n!}\sum_{|I|=|J|=n-p}\int_{(y,\zeta)\in\rb^n\times\rb^n}\eta(y)\delta_0(x-y)T_{IJ}(y)dy_I\wedge d\zeta_J\wedge\beta^n(x-y,\xi-\zeta).$$
Since $\beta^n(x-y,\xi-\zeta)=\bigwedge_{j=1}^{n}(dx_j\wedge d\xi_j-dx_j\wedge d\zeta_j-dy_j\wedge d\xi_j+dy_j\wedge d\zeta_j)$, then by induction on $q=n-p$, the $(n,n)$-component in $(y,\zeta)$ and the $(n-p,n-p)$-component in $(x,\xi)$ of the form $dy_I\wedge d\zeta_J\wedge\beta^{n}(x-y,\xi-\zeta)$ is equals to
$$\frac{n!}{(n-p)!p!}dy_I\wedge d\zeta_J\wedge\beta^{n-p}(x,\xi)\wedge\beta^p(y,\zeta)=dx_I\wedge d\xi_J\wedge\beta^n(y,\zeta)=n!dx_I\wedge d\xi_J\wedge\beta_n(y,\zeta).$$
Thus,
$$\begin{array}{lcl}
f_j(x,\xi)&=&\ds\sum_{|I|=|J|=n-p}dx_I\wedge d\xi_J\left(\int_{(y,\zeta)\in\rb^n\times\rb^n}\chi_j(x-y)\eta(y)T_{IJ}(y)\beta_n(y,\zeta)\right)\\&=&\ds\sum_{|I|=|J|=n-p}(\eta T_{IJ}\ast\chi_j)(x)dx_I\wedge d\xi_J=(\eta T\ast\chi_j)(x,\xi)
\end{array}$$ and
$$\begin{array}{lcl}
g(x,\xi)&=&\ds\sum_{|I|=|J|=n-p}dx_I\wedge d\xi_J\left(\int_{(y,\zeta)\in\rb^n\times\rb^n}\delta_0(x-y)\eta(y)T_{IJ}(y)\beta_n(y,\zeta)\right)\\&=&\ds\sum_{|I|=|J|=n-p}(\eta T_{IJ}\ast\delta_0)(x)dx_I\wedge d\xi_J=\eta T(x,\xi).
\end{array}$$
Hence, $dd^{\#}U_j=(\eta T)\ast\chi_j+R_j$ and $dd^{\#}U=\eta T+R$ on $\rb^n\times\rb^n$. In particular, we obtain $dd^{\#}U=T+R$ on $\Omega\times\rb^n$.
\end{proof}
In the following theorem, we use the local potential as a tool to extend Theorem 6.1 in \cite{18} for the case where $T$ is weakly positive and $T\wedge\beta^{n-p}$ is convex.
\begin{thm} Let $\Omega$ be an open subset of $\rb^n$ and  $T$ be a weakly positive convex current of bidimension $(p,p)$ on
$\{\Omega\smallsetminus K\}\times\rb^n$ with locally finite mass near $K$, where $K$ is a compact subset of $\rb^n$ with sigma-finite $(p-2)$-dimensional Hausdorff measure.  Then, there exists a positive measure $S$ such that
$$\widetilde{dd^{\#}T\wedge\beta^{p-1}}=dd^{\#}\widetilde{T}\wedge\beta^{p-1}+S,$$ where $\widetilde{T}$ and $\widetilde{dd^{\#}T\wedge\beta^{p-1}}$ are the  extensions by $0$ of $T$ and $dd^{\#}T\wedge\beta^{p-1}$ across $K\times\rb^n$.
\end{thm}
\begin{proof} We prove firstly that $dd^{\#}T\wedge\beta^{p-1}$ is of locally finite mass near $K\times\rb^n$. Consider $U$ as the local potential associated to $dd^{\#}T$. Since the problem is local, we can assume that $\Omega=\{\psi<0\}$ is a strictly convex domain on $\rb^n$, where $\psi$ is a smooth strictly convex exhaustion function on $\Omega$. As in Lemma \ref{lmm1}, we have $dd^{\#}U=dd^{\#}T+R$, where $R$ is with uniformly bounded coefficients on $\Omega\times\rb^n$, then there exists $A>0$ such that $R\geqslant -A(dd^{\#}\psi)^{n-p+1}$. Setting $W=U-T+A\psi(dd^{\#}\psi)^{n-p}$, then it is clear that $W$ is a weakly negative convex current of bidimension $(p,p)$ on $\{\Omega\smallsetminus K\}\times\rb^n$. Since $W$ has a locally finite mass near $K\times\rb^n$, then by applying Theorem 6.1 in \cite{18} for the current $-W$, there exists a positive measure $S_1$ such that 
\begin{equation}\label{E2}
\widetilde{dd^{\#}W\wedge\beta^{p-1}}+S_1=dd^{\#}\widetilde{W}\wedge\beta^{p-1}.
\end{equation}
 On the other hand, we have $$dd^{\#}\widetilde{W}\wedge\beta^{p-1}=dd^{\#}U\wedge\beta^{p-1}-dd^{\#}\widetilde{T}\wedge\beta^{p-1}+A(dd^{\#}\psi)^{n-p+1}\wedge\beta^{p-1}$$ and 
$$\widetilde{dd^{\#}W\wedge\beta^{p-1}}=\widetilde{dd^{\#}U\wedge\beta^{p-1}}-\widetilde{dd^{\#}T\wedge\beta^{p-1}}+A(dd^{\#}\psi)^{n-p+1}\wedge\beta^{p-1}.$$
Next, by applying once again Theorem 6.1 in \cite{18} on the current $-U-A\psi(dd^{\#}\psi)^{n-p}$ which is weakly positive and concave on $\{\Omega\smallsetminus K\}\times\rb^n$, there exists a positive measure $S_2$ such that 
$$\widetilde{dd^{\#}U\wedge\beta^{p-1}}+A(dd^{\#}\psi)^{n-p+1}\wedge\beta^{p-1}+S_2=dd^{\#}U\wedge\beta^{p-1}+A(dd^{\#}\psi)^{n-p+1}\wedge\beta^{p-1}.$$
Consequently, in view of $(\ref{E2})$, the proof is completed by considering $S=S_1-S_2$ which is a positive measure due to the proof of the Theorem 6.1 in \cite{18}.
\end{proof}
\subsection{SuperHessian operator}
In this subsection, we will focus on the convergence of the sequence of operators $U_j\wedge dd^{\#}v^{j}_1\wedge...\wedge dd^{\#}v^{j}_q$, where $U_j =U\ast\chi_j$ is a smooth regularization of the local potential $U$ and $(v^{j}_1)_j,...,(v^{j}_q)_j$ are sequences of convex functions which converge locally uniformly respectively towards $v_1,...,v_q$. Strongly motivated by techniques goes back to Ben Messaoud and El Mir \cite{12} in the complex analysis, we extend firstly to the superformalism context the following convergence result: 
\begin{thm A} Let $1\leqslant p\leqslant n$, $S$ is a current of bidimension $(p,p)$ on $\Omega\times\rb^n$ and $v_1,...,v_q$ are convex functions on $\Omega$ for $1\leqslant q\leqslant p$. If we assume that there exists $(S_j)_j$ a sequence of smooth forms  of bidegree $(n-p,n-p)$ on $\Omega\times\rb^n$ such that $S_j$ is negative, $dd^{\#}S_j$ is positive and $(S_j)_j$ decreases weakly to $S$, and if we take $v^{j}_1,...,v^{j}_q$ to be sequences of convex functions which converge locally uniformly respectively to $v_1,...,v_q$. Then, we have:
\begin{enumerate}
\item $(S_j\wedge dd^{\#}v_1\wedge...\wedge dd^{\#}v_q)_j$ converges weakly on $\Omega$, where $j\rightarrow+\infty$. Denote its limit by $S\wedge dd^{\#}v_1\wedge...\wedge dd^{\#}v_q$.
\item If we also assume that $v^{j}_i\geqslant v_i, \forall i=1,...,q$, then:
\begin{itemize}
\item[i)] $S\wedge dd^{\#}v^{j}_1\wedge...\wedge dd^{\#}v^{j}_q\longrightarrow S\wedge dd^{\#}v_1\wedge...\wedge dd^{\#}v_q$ weakly on $\Omega$.
\item[ii)] $S_j\wedge dd^{\#}v^{j}_1\wedge...\wedge dd^{\#}v^{j}_q\longrightarrow S\wedge dd^{\#}v_1\wedge...\wedge dd^{\#}v_q$ weakly on $\Omega$.
\item[iii)] $dd^{\#}(S\wedge dd^{\#}v_1\wedge...\wedge dd^{\#}v_q)=dd^{\#}S\wedge dd^{\#}v_1\wedge...\wedge dd^{\#}v_q$.
\end{itemize}
\end{enumerate}
\end{thm A}
\begin{exe} The current $S=\frac{-\beta^{p+1}}{|x|^{p}}$ satisfies the hypothesis of Theorem A. Indeed, if we consider the sequence 
$S_j=\frac{-\beta^{p+1}}{\left(|x|^{2}+\frac{1}{j}\right)^{\frac{p}{2}}}\leqslant0$, a simple computation gives 
$$\begin{array}{lcl}
dd^{\#}S_j&=&\ds p\left(\frac{\beta^{p+2}}{\left(|x|^{2}+\frac{1}{j}\right)^{\frac{p+2}{2}}}-\frac{p+2}{4}\frac{d|x|^2\wedge d^{\#}|x|^2\wedge\beta^{p+1}}{\left(|x|^{2}+\frac{1}{j}\right)^{\frac{p+3}{2}}}\right)\\&=&\ds p\left(dd^{\#}\left(\left(|x|^{2}+\frac{1}{j}\right)^{\frac{1}{2}}\right)\right)^{p+2}
\geqslant0.
\end{array}$$ 
\end{exe}
The proof of Theorem A needs some tools such as the Chern-Levine-Nirenberg type inequality and two lemmas in the superformalism setting with proofs almost identical to the complex case.
\begin{pro}[Chern-Levine-Nirenberg type inequality]\label{pp3} Let $T$ be a positive closed current of bidimension $(p,p)$ on $\Omega\times\rb^n$ and let $v_1,..,v_q$ are convex functions on $\Omega$ for $1\leqslant q\leqslant p$. Then, for every compact $K$ and $L$ with $K\subset\mathring{L}\Subset\Omega$, there exists a constant $C_{K,L}\geqslant 0$ such that
$$\|T\wedge dd^{\#}v_1\wedge...\wedge dd^{\#}v_q\|_K\leqslant C_{K,L}\|v_1\|_{L^\infty(L)}...\|v_q\|_{L^\infty(L)}\|T\|_{L}.$$
\end{pro}
\begin{proof} Using induction, it suffices to prove that
$$\|T\wedge dd^{\#}v\|_K\leqslant C_{K,L}\|v\|_{L^\infty(L)}\|T\|_{L}.$$
In fact, let $\chi$ be a smooth and compactly supported function on $L$ such that $0\leqslant\chi\leqslant1$ and $\chi\equiv1$ on $K$. Thus, Proposition 4.1 in \cite{5} combined with an integration by parts, yield
$$\begin{array}{lcl}
\|T\wedge dd^{\#}v\|_K&\leqslant&\ds C_1\int_{K\times\rb^n}T\wedge\beta^{p-1}\wedge dd^{\#}v\\&\leqslant&\ds C_1\int_{L\times\rb^n}\chi T\wedge\beta^{p-1}\wedge dd^{\#}v\\&=&\ds C_1\int_{L\times\rb^n}vT\wedge\beta^{p-1}\wedge dd^{\#}\chi\\&\leqslant&\ds C_2\|v\|_{L^\infty(L)}\|T\|_{L},
\end{array}$$
where $C_2$ depends only on bounds of coefficients of $dd^{\#}\chi$ and on the set $L$.
\end{proof}
\begin{lem}\label{lmmm1} Let $\varphi$ be a symmetric (closed) form of bidegree $(p,p)$ on $\rb^n\times\rb^n$. Then, there exist two positive (closed) forms $\varphi_{1}$ and $\varphi_{2}$ of bidegree $(p,p)$ such that $\varphi=\varphi_{1}-\varphi_{2}$.
\end{lem}
\begin{lem}\label{lmmm2} Let $\Omega_0\Subset\Omega$ and $1\leqslant q\leqslant n$. Let $\varphi$ be a form of bidegree $(n-q,n-q)$ on $\Omega_0\times\rb^n$ such that $\varphi$ is smooth on $\overline{\Omega}_0\times\rb^n$ and $dd^{\#}\varphi$ is positive, and let $v_1,...,v_q$ and $w_1,...,w_q$ are convex functions on a neighborhood of $\overline{\Omega}_0$ , such that:
\begin{itemize}
\item[i)] $\forall i=1,...,q,\ w_i\geqslant v_i$ on $\Omega_0$,
\item[ii)] $\forall i=1,...,q,\ w_i=v_i$ on a neighborhood of $\partial\Omega_0$.
\end{itemize}
Then, $$\ds\int_{\Omega_0\times\rb^n}\varphi\wedge dd^{\#}w_1\wedge...\wedge dd^{\#}w_q\geqslant\ds\int_{\Omega_0\times\rb^n}\varphi\wedge dd^{\#}v_1\wedge...\wedge dd^{\#}v_q.$$
\end{lem}

\begin{proofof}{\it Theorem A.} Since the problem is locally, then we can work on a strictly convex subset $\Omega_0\Subset\Omega$ such that $\Omega_0=\{\psi<0\}$. Thanks to Theorem B in \cite{22}, we chose $\psi$ to be a smooth strictly convex exhaustion function on $\overline{\Omega}_0$ such that $dd^{\#}(\psi-|x|^2)\geqslant0$. For $i=1,...,q$, the function $v_i$ is convex on $\overline{\Omega}_0$, then it is locally bounded and we can assume that $-M\leqslant v_i\leqslant-1$, where $M>0$.

(1) Let $\varphi$ be a symmetric form of bidegree $(p-q,p-q)$ on $\Omega_0\times\rb^n$, we need to prove that the sequence $(\langle S_j\wedge dd^{\#}v_1\wedge...\wedge dd^{\#}v_q,\varphi\rangle)_j$ converges. Thanks to Lemma \ref{lmmm1}, it suffices to suppose that $\varphi$ is positive. So that $(\langle S_j\wedge dd^{\#}v_1\wedge...\wedge dd^{\#}v_q,\varphi\rangle)_j$ is decreasing and it remains to show that it is bounded from below. Let $K\Subset\Omega_0$ such that $\Supp\varphi\subset K$, and let $c>0$ such that $\varphi\leqslant c{\1}_K(dd^{\#}\psi)^{p-q}$. Since $S_j\leqslant0$, then
$$\int_{\Omega_0\times\rb^n}S_j\wedge dd^{\#}v_1\wedge...\wedge dd^{\#}v_q\wedge\varphi\geqslant c\int_{K\times\rb^n}S_j\wedge dd^{\#}v_1\wedge...\wedge dd^{\#}v_q\wedge(dd^{\#}\psi)^{p-q}.$$
Let $\delta>0$ small enough such that $K\subset\Omega_\delta=\{\psi<-\delta\}$, and let $w_i=\max\left(\frac{M}{\delta}\psi,v_i\right),\ \forall i=1,...,q$. So, $w_i$ is convex in a neighborhood of $\overline{\Omega}_0$, $w_i=v_i$ on $\Omega_\delta$ and $w_i=\frac{M}{\delta}\psi$ on $\Omega_0\smallsetminus\Omega_{\frac{\delta}{M}}$. Thus,
$$\begin{array}{lcl}
& &\ds\int_{K\times\rb^n}S_j\wedge dd^{\#}v_1\wedge...\wedge dd^{\#}v_q\wedge(dd^{\#}\psi)^{p-q}\geqslant\\&\geqslant&\ds\int_{\Omega_\delta\times\rb^n}S_j\wedge dd^{\#}w_1\wedge...\wedge dd^{\#}w_q\wedge(dd^{\#}\psi)^{p-q}\\&\geqslant&\ds\int_{\Omega_{\frac{\delta}{2M}}\times\rb^n}S_j\wedge dd^{\#}w_1\wedge...\wedge dd^{\#}w_q\wedge(dd^{\#}\psi)^{p-q}.
\end{array}$$
Let $\chi$ be a positive and compactly supported function on $\Omega_{\frac{\delta}{3M}}$ such that $\chi=-\psi$ on $\Omega_{\frac{\delta}{2M}}$, then
$$\begin{array}{lcl}
& &\ds\int_{\Omega_{\frac{\delta}{2M}}\times\rb^n}S_j\wedge dd^{\#}w_1\wedge...\wedge dd^{\#}w_q\wedge(dd^{\#}\psi)^{p-q}=\\&=&\ds-\int_{\Omega_{\frac{\delta}{2M}}\times\rb^n}S_j\wedge dd^{\#}w_1\wedge...\wedge dd^{\#}w_q\wedge(dd^{\#}\psi)^{p-q-1}\wedge dd^{\#}\chi\\&=&\ds-\int_{\Omega_{\frac{\delta}{3M}}\times\rb^n}\chi dd^{\#}S_j\wedge dd^{\#}w_1\wedge...\wedge dd^{\#}w_q\wedge(dd^{\#}\psi)^{p-q-1}\\&+&\ds\frac{M^q}{\delta^q}\int_{(\Omega_{\frac{\delta}{3M}}\smallsetminus\Omega_{\frac{\delta}{2M}})\times\rb^n}S_j\wedge(dd^{\#}\psi)^{p-1}\wedge dd^{\#}\chi.
\end{array}$$
By using Proposition \ref{pp3}, we have
$$\left|\int_{\Omega_{\frac{\delta}{2M}}\times\rb^n}S_j\wedge dd^{\#}w_1\wedge...\wedge dd^{\#}w_q\wedge(dd^{\#}\psi)^{p-q}\right|\leqslant M^qC_{K,\Omega_0}\left(\|dd^{\#}S_j\|_{\Omega_{\frac{\delta}{3M}}}+\|S_j\|_{\Omega_{\frac{\delta}{3M}}}\right).$$
Now, let $\chi'$ be a smooth and compactly supported function on $\Omega_0$ such that $0\leqslant\chi\leqslant1$ and $\chi\equiv1$ on $\overline{\Omega}_{\frac{\delta}{3M}}$. Since $dd^{\#}S_j$ is positive, Proposition 4.1 in \cite{5} and an integration by parts, yield
$$\begin{array}{lcl}
\ds\|dd^{\#}S_j\|_{\Omega_{\frac{\delta}{3M}}}&\leqslant&\ds C\int_{\Omega_{\frac{\delta}{3M}}\times\rb^n}dd^{\#}S_j\wedge\beta^{n-p-1}\\&\leqslant&\ds C\int_{\Omega_{0}\times\rb^n}\chi' dd^{\#}S_j\wedge\beta^{n-p-1}\\&\leqslant&\ds C\int_{\Omega_{0}\times\rb^n} S_j\wedge dd^{\#}\chi'\wedge\beta^{n-p-1}\leqslant C_{\Omega_0}\|S_j\|_{\Omega_0}.
\end{array}$$
Thus, since $0\geqslant S_j\geqslant S$, we have
\begin{equation}\label{eqq1}
\left|\int_{K\times\rb^n}S_j\wedge dd^{\#}v_1\wedge...\wedge dd^{\#}v_q\wedge(dd^{\#}\psi)^{p-q}\right|\leqslant  M^qC_{K,\Omega_0}\|S\|_{\Omega_0},
\end{equation}
and (1) will follows.

(2)i) Firstly, we assume that $q<p$. Let $\Theta$ be a limit of the sequence $(S\wedge dd^{\#}v^{j}_1\wedge...\wedge dd^{\#}v^{j}_q)_j$, which is locally bounded in mass thanks to (\ref{eqq1}). To prove that $\Theta=S\wedge dd^{\#}v_1\wedge...\wedge dd^{\#}v_q$, we take a sub-sequence of $(S\wedge dd^{\#}v^{j}_1\wedge...\wedge dd^{\#}v^{j}_q)_j$ which converges weakly to $\Theta$. Let $k>l\geqslant1$, then
$$S_k\wedge dd^{\#}v^{j}_1\wedge...\wedge dd^{\#}v^{j}_q\leqslant S_l\wedge dd^{\#}v^{j}_1\wedge...\wedge dd^{\#}v^{j}_q.$$
When we let $k\rightarrow+\infty$, $j\rightarrow+\infty$ and $l\rightarrow+\infty$ in this order, the first result (1) yields
\begin{equation}\label{eqq2}
\Theta\leqslant S\wedge dd^{\#}v_1\wedge...\wedge dd^{\#}v_q.
\end{equation}
Since every $(v_{i}^{j})_j$ converges locally uniformly to $v_i$, we can assume, without loss of generality, that $\forall i=1,...,q$ and $\forall j\in\nb^\ast$, $-M\leqslant v_{i}^{j}\leqslant-1$ on $\overline{\Omega}_0$. Let $K\Subset\Omega_0$ and let $\delta$ small enough such that $K\Subset\Omega_\delta$. Let us set $w^{j}_i=\max\left(\frac{M}{\delta}\psi,v^{j}_i\right)$ and $w_i=\max\left(\frac{M}{\delta}\psi,v_i\right),\ \forall i=1,...,q$. So, $w^{j}_i$ is convex in a neighborhood of $\overline{\Omega}_0$, $w^{j}_i=v^{j}_i$ on $\Omega_\delta$ and $w^{j}_i=\frac{M}{\delta}\psi$ on $\Omega_0\smallsetminus\Omega_{\frac{\delta}{M}}$. Let $\Theta'$ be an another limit of $(S\wedge dd^{\#}v^{j}_1\wedge...\wedge dd^{\#}v^{j}_q)_j$ on a neighborhood of $\overline{\Omega}_0\times\rb^n$, as in (\ref{eqq2}) we have
\begin{equation}\label{eqq3}
\Theta'\leqslant S\wedge dd^{\#}w_1\wedge...\wedge dd^{\#}w_q.
\end{equation}
Lemma \ref{lmmm2} implies
$$\int_{\Omega_0\times\rb^n}S_l\wedge dd^{\#}w^{j}_1\wedge...\wedge dd^{\#}w^{j}_q\wedge(dd^{\#}\psi)^{p-q}\geqslant\int_{\Omega_0\times\rb^n}S_l\wedge dd^{\#}w_1\wedge...\wedge dd^{\#}w_q\wedge(dd^{\#}\psi)^{p-q},$$
and thanks to (1), we get
\begin{equation}\label{eqq4}
\int_{\Omega_0\times\rb^n}S\wedge dd^{\#}w^{j}_1\wedge...\wedge dd^{\#}w^{j}_q\wedge(dd^{\#}\psi)^{p-q}\geqslant\int_{\Omega_0\times\rb^n}S\wedge dd^{\#}w_1\wedge...\wedge dd^{\#}w_q\wedge(dd^{\#}\psi)^{p-q}.
\end{equation}
Now, we replace the sequence of positive measures $(S\wedge dd^{\#}w^{j}_1\wedge...\wedge dd^{\#}w^{j}_q\wedge(dd^{\#}\psi)^{p-q})_j$ with a sub-sequence $\eta_j$ which converges weakly to $\eta=\Theta'\wedge(dd^{\#}\psi)^{p-q}$ on $\overline{\Omega}_0\times\rb^n$. Since $\limsup_{j\rightarrow+\infty}\eta_j(\Omega_0)\leqslant\eta(\Omega_0)$, then (\ref{eqq4}) yields
$$\int_{\Omega_0\times\rb^n}\Theta'\wedge(dd^{\#}\psi)^{p-q}\geqslant\int_{\Omega_0\times\rb^n}S\wedge dd^{\#}w_1\wedge...\wedge dd^{\#}w_q\wedge(dd^{\#}\psi)^{p-q}.$$
Thus, by (\ref{eqq3}), we have
\begin{equation}\label{eqq5}
\int_{\Omega_0\times\rb^n}\Theta'\wedge(dd^{\#}\psi)^{p-q}=\int_{\Omega_0\times\rb^n}S\wedge dd^{\#}w_1\wedge...\wedge dd^{\#}w_q\wedge(dd^{\#}\psi)^{p-q}.
\end{equation}
Let $\varphi$ be a symmetric form of bidegree $(p-q-1,p-q-1)$ on $\Omega_0\times\rb^n$, then according to Lemma \ref{lmmm1}, there exist two positive closed forms $\varphi_{1}$ and $\varphi_{2}$ of bidegree $(p-q,p-q)$ on $\Omega_0\times\rb^n$ such that $dd^{\#}\varphi=\varphi_1-\varphi_2$. Thus, (\ref{eqq5}) remains valid for $\varphi_1$ and $\varphi_2$ instead of $(dd^{\#}\psi)^{p-q}$, and then it is valid for $dd^{\#}\varphi$ which implies that $dd^{\#}E=0$. Since $E=S\wedge dd^{\#}w_1\wedge...\wedge dd^{\#}w_q-\Theta'$ is positive and is compactly supported, then Proposition 4.2 in \cite{18} yields $E=0$, and therefore $(S\wedge dd^{\#}w^{j}_1\wedge...\wedge dd^{\#}w^{j}_q)_j$ converges weakly to $S\wedge dd^{\#}w_1\wedge...\wedge dd^{\#}w_q$. Since $w^{j}_i=v^{j}_i$ and $w_i=v_i$ on $K$, then $(S\wedge dd^{\#}v^{j}_1\wedge...\wedge dd^{\#}v^{j}_q)_j$ converges weakly to $S\wedge dd^{\#}v_1\wedge...\wedge dd^{\#}v_q$ on $K\times\rb^n$, which concludes (2)i). Secondly, for $q=p$, let $\pi$ be the projection of $\Omega\times\rb$ on $\Omega$. Then, the currents $\pi^\ast(S)$ and $\pi^\ast(S_j)$ are of bidimension $(p+1,p+1)$ and verifies the hypothesis of Theorem A. Since $p+1>q$, then we can use the first step by setting $\tilde{v}_i=v_i\circ\pi$ and $\tilde{v}^{j}_i=v^{j}_i\circ\pi$, for all $i=1,...,q$, and our result will be concluded by Fubini's theorem.

ii) Let $v_{i}^{j,k}=v_{i}^{j}\ast\chi_k$ be the standard regularization of $v_{i}^{j}$. Since $dd^{\#}v^{j,k}_1\wedge...\wedge dd^{\#}v^{j,k}_q$ is positive and the sequence $(S_j)_j$ is decreasing, we have $$S_j\wedge dd^{\#}v^{j,k}_1\wedge...\wedge dd^{\#}v^{j,k}_q\leqslant S_l\wedge dd^{\#}v^{j,k}_1\wedge...\wedge dd^{\#}v^{j,k}_q,\quad\forall j>l.$$
Furthermore, if $k\rightarrow+\infty$, then
$$S_j\wedge dd^{\#}v^{j}_1\wedge...\wedge dd^{\#}v^{j}_q\leqslant S_l\wedge dd^{\#}v^{j}_1\wedge...\wedge dd^{\#}v^{j}_q,\quad\forall j>l.$$
Since the sequence $(S_j\wedge dd^{\#}v^{j}_1\wedge...\wedge dd^{\#}v^{j}_q)_j$ is locally bounded in mass, then we can extract a weakly convergent sub-sequence and denote by $\Theta$ his limit. Then,
$$\Theta\leqslant S_l\wedge dd^{\#}v_1\wedge...\wedge dd^{\#}v_q.$$
By (1), when $l\rightarrow+\infty$, we get
$$\Theta\leqslant S\wedge dd^{\#}v_1\wedge...\wedge dd^{\#}v_q.$$
Next, since $S_j\geqslant S$, we have
$$S_j\wedge dd^{\#}v^{j,k}_1\wedge...\wedge dd^{\#}v^{j,k}_q\geqslant S\wedge dd^{\#}v^{j,k}_1\wedge...\wedge dd^{\#}v^{j,k}_q.$$
By (2)i), when $k\rightarrow+\infty$ and $j\rightarrow+\infty$ respectively, we get
$$\Theta\geqslant S\wedge dd^{\#}v_1\wedge...\wedge dd^{\#}v_q.$$

iii) Here, all we need to do is to replace $v_i$ with the standard regularization $v_{i}\ast\chi_k$, and just use (2)i) and the continuity of the operator $dd^{\#}$.
\end{proofof}

As an immediate consequence of Theorem A, we get our main theorem of this section which is the counterpert of a result due to \cite{12} in the complex setting:
\begin{thm B} Let $1\leqslant p<n$ and let $v_1,..,v_q$ are convex functions on $\Omega$ for $1\leqslant q\leqslant p+1$. Let $U$ be the local potential associated to $T$ a positive closed current of bidimension $(p,p)$ on $\rb^n\times\rb^n$ and $U_j=U\ast\chi_j$. If we take $v^{j}_1,...,v^{j}_q$ to be sequences of convex functions which converge locally uniformly respectively to $v_1,...,v_q$. Then, we have:
\begin{enumerate}
\item $(U_j\wedge dd^{\#}v_1\wedge...\wedge dd^{\#}v_q)_j$ converges weakly on $\Omega$, where $j\rightarrow+\infty$. Denote its limit by $U\wedge dd^{\#}v_1\wedge...\wedge dd^{\#}v_q$.
\item If we also assume that $v^{j}_i\geqslant v_i, \forall i=1,...,q$, then:
\begin{itemize}
\item[i)] $U\wedge dd^{\#}v^{j}_1\wedge...\wedge dd^{\#}v^{j}_q\longrightarrow U\wedge dd^{\#}v_1\wedge...\wedge dd^{\#}v_q$ weakly on $\Omega$.
\item[ii)] $U_j\wedge dd^{\#}v^{j}_1\wedge...\wedge dd^{\#}v^{j}_q\longrightarrow U\wedge dd^{\#}v_1\wedge...\wedge dd^{\#}v_q$ weakly on $\Omega$.
\item[iii)] $dd^{\#}(U\wedge dd^{\#}v_1\wedge...\wedge dd^{\#}v_q)=dd^{\#}U\wedge dd^{\#}v_1\wedge...\wedge dd^{\#}v_q$.
\end{itemize}
\end{enumerate}
\end{thm B}
\begin{proof} Since the problem is locally, then we can work on a strictly convex subset $\Omega_0\Subset\Omega$ such that $\Omega_0=\{\psi<0\}$, where $\psi$ is a smooth strictly convex exhaustion function on $\overline{\Omega}_0$. As in Lemma \ref{lmm1}, $R_j$ is with uniformly bounded coefficients on $\Omega_0\times\rb^n$, then there exists $A>0$ such that $dd^{\#}(U_j+A\psi(dd^{\#}\psi)^{n-p-1})$ is positive. Finally, the proof is completed by using \cite{5} and applying Theorem A with $S_j=U_j+A\psi(dd^{\#}\psi)^{n-p-1}$. 
\end{proof}
As an application and in view of Lemma \ref{lmm1}, we obtain:
\begin{cor}\label{cc1} Let $T$ be a positive closed current of bidimension $(p,p)$ on $\rb^n\times\rb^n$, $T_j=T\ast\chi_j$ and $v_0,..,v_q$ are convex functions on $\Omega$ for $1\leqslant q\leqslant p$. If we take $v^{j}_0,...,v^{j}_q$ to be sequences of convex functions which converge locally uniformly respectively to $v_0,...,v_q$ such that $v^{j}_i\geqslant v_i, \forall i=1,...,q$. Then, we have:
\begin{enumerate}
\item $T_j\wedge dd^{\#}v^{j}_1\wedge...\wedge dd^{\#}v^{j}_q\longrightarrow T\wedge dd^{\#}v_1\wedge...\wedge dd^{\#}v_q$ weakly on $\Omega$.
\item $v^{j}_0T_j\wedge dd^{\#}v^{j}_1\wedge...\wedge dd^{\#}v^{j}_q\longrightarrow v_0T\wedge dd^{\#}v_1\wedge...\wedge dd^{\#}v_q$ weakly on $\Omega$.
\end{enumerate}
\end{cor}
Notice that Corollary \ref{cc1} is the corresponding result of the one given by \cite{12} in complex analysis for plurisubharmonic functions and closed positive currents. 
\begin{proof} Since the problem is locally, then we can work on a strictly convex subset $\Omega_0\Subset\Omega$ such that $\Omega_0=\{\psi<0\}$, where $\psi$ is a smooth strictly convex exhaustion function on $\overline{\Omega}_0$.

(1)  Let $U$ be the local potential associated to $T$ on $\Omega_0\times\rb^n$ and $U_j=U\ast\chi_j$. By using Theorem B, Lemma \ref{lmm1} and the continuity of the operator $dd^{\#}$, we obtain
$$T_j\wedge dd^{\#}v^{j}_1\wedge...\wedge dd^{\#}v^{j}_q+R_j\wedge dd^{\#}v^{j}_1\wedge...\wedge dd^{\#}v^{j}_q\longrightarrow dd^{\#}U\wedge dd^{\#}v_1\wedge...\wedge dd^{\#}v_q.$$
As $R$ is with uniformly bounded coefficients on $\Omega_0\times\rb^n$, then by \cite{5}, we have $$R\wedge dd^{\#}v^{j}_1\wedge...\wedge dd^{\#}v^{j}_q\longrightarrow R\wedge dd^{\#}v_1\wedge...\wedge dd^{\#}v_q.$$
Then,
$$T_j\wedge dd^{\#}v^{j}_1\wedge...\wedge dd^{\#}v^{j}_q+(R_j-R)\wedge dd^{\#}v^{j}_1\wedge...\wedge dd^{\#}v^{j}_q\longrightarrow T\wedge dd^{\#}v_1\wedge...\wedge dd^{\#}v_q.$$
Since $(R_j-R)_j$ converges uniformly to  $0$ on $\Omega_0\times\rb^n$, we get
$$T_j\wedge dd^{\#}v^{j}_1\wedge...\wedge dd^{\#}v^{j}_q\longrightarrow T\wedge dd^{\#}v_1\wedge...\wedge dd^{\#}v_q.$$

(2) For $q<p$, let $\Theta$ be a weak limit of $(v^{j}_0T_j\wedge dd^{\#}v^{j}_1\wedge...\wedge dd^{\#}v^{j}_q)_j$. Thus, by regularizing $v_0$ and $v^{j}_0$, a simple computation gives $\Theta\leqslant v_0T\wedge dd^{\#}v_1\wedge...\wedge dd^{\#}v_q$. Moreover, in a similar way as in the proof of Theorem A (2)i), the current $E=v_0T\wedge dd^{\#}v_1\wedge...\wedge  dd^{\#}v_q-\Theta$ of bidimension $(p-q,p-q)$ on $\Omega_0\times\rb^n$ is compactly supported and positive. Hence, according to previous argument, we have $dd^{\#}E=0$, and thus $E=0$. In case $q=p$, we place ourselves in $\Omega\times\rb$ and we proceed as the previous case, then we conclude our proof thanks to Fubini's theorem.
\end{proof}
\begin{rem} It should be noted that Corollary \ref{cc1} can be obtained directly without passing by Theorem A and for a general sequence of positive closed currents $(T_j)_j$ converging weakly to $T$. In fact, it suffices to adapt Proposition 3.2 in \cite{18} to our situation. In particular, when $T_j=T$ we recover a result goes back to Lagerberg \cite{5}. 
\end{rem}
\subsection{Lelong number of a local potential}
Let us start by defining the Euclidean sphere and ball of centre $a$ and radius $r$ respectively by
$$\mathbb{S}(a,r)=\{x\in\Omega;\ |x-a|=r\}\qquad\mathrm{and}\qquad \mathbb{B}(a,r)=\{x\in\Omega;\ |x-a|<r\}.$$
Next, for every weakly positive current $T$ of bidimension $(p,p)$ on $\rb^n\times\rb^n$, we define the Lelong number of $T$ at $a$ (when it exists) by $$\nu_T(a)=\lim_{r\rightarrow0}\nu_T(a,r),\ {\rm where}\ \nu_T(a,r)=\frac{1}{r^p}\int_{\mathbb{B}(a,r)\times\mathbb{R}^n}T\wedge\beta^{p}.$$ If $T$ is assumed to be closed, it was proved by \cite{5} that $\nu_T(a)$ exists for any $a$. In analogue with the work of Ghiloufi, Zaway and Hawari \cite{0}, the next result describe a relation between the Lelong number of a positive closed current and the Lelong number of his local potential.
\begin{thm}\label{1} Let $T$ be a weakly positive closed current of bidimension $(p,p)$ on $\rb^n\times\rb^n$ with measure coefficients and $1\leqslant p\leqslant n-1$. Let $\Omega\Subset\rb^n$ and $\eta$ be a smooth and compactly supported function on $\rb^n$ such that $0\leqslant\eta\leqslant1$ and $\eta\equiv1$ on $\overline{\Omega}$. If we assume that the Lelong number of $T$ vanishes at every point of $\Omega$, then the Lelong number of the local potential $U=U(\eta,T)$ at every point of $\Omega$ exists and is equal to zero.
\end{thm}
If $p=n-1$, Theorem \ref{1} is obvious. Indeed, $T=dd^{\#}u$ for a convex function and by Example \ref{x1}, we can take $U=u$ (modulo a smooth function), so we have the equality between the Lelong numbers of $T$ and $U$.
\begin{proof} Let $x_0\in\Omega$ and $0<r<d(x_0,\partial\Omega)$. The calculus in the proof of Proposition \ref{pro1}, yields
$$\begin{array}{lcl}
-\nu_U(x_0,r)&=&\ds\frac{c_n}{r^{p+1}}\int_{(x,\xi)\in \mathbb{B}(x_0,r)\times\rb^n}\int_{(y,\zeta)\in\rb^n\times\rb^n}\eta(y)T(y,\zeta)\wedge\frac{\beta^{n-1}(x-y,\xi-\zeta)}{|x-y|^{n-2}}\wedge\beta^{p+1}(x,\xi)\\&=&\ds\frac{n!c_n}{r^{p+1}}\int_{(y,\zeta)\in\rb^n\times\rb^n}\eta(y)T(y,\zeta)\wedge\beta^{p}(y,\zeta)\int_{x\in \mathbb{B}(x_0,r)}\frac{1}{|x-y|^{n-2}}d\lambda,
\end{array}$$
where $d\lambda=dx_1\wedge...\wedge dx_n$. Let $d\sigma=\sum_{i=1}^{n}(-1)^{i-1}dx_1\wedge...\wedge dx_{i-1}\wedge dx_{i+1}\wedge...\wedge dx_n$ and $\mathbb{S}=\mathbb{S}(0,1)$, then we have
$$\begin{array}{lcl}
\ds\int_{x\in \mathbb{B}(x_0,r)}\frac{1}{|x-y|^{n-2}}d\lambda&=&\ds\int_{x\in \mathbb{B}(x_0,r)}\frac{1}{|(x-x_0)-(y-x_0)|^{n-2}}d\lambda\\&=&\ds\int_{0}^{r}\int_{\mathbb{S}}\frac{t^{n-1}}{|tu-(y-x_0)|^{n-2}}dtd\sigma(u)\\&=&\ds\int_{0}^{r}t\left(\int_{\mathbb{S}}\frac{d\sigma(u)}{|u-\frac{y-x_0}{t}|^{n-2}}\right)dt\\&=&\ds\int_{0}^{r}t\min\left(1,\left(\frac{t}{|y-x_0|}\right)^{n-2}\right)dt.
\end{array}$$
Assume first that $T$ is smooth, then for $r<r_0<d(x_0,\partial\Omega)$ and $\sigma_T(x_0,t)=\int_{(y,\zeta)\in\mathbb{B}(x_0,t)\times\rb^n}T(y,\zeta)\wedge\beta^{p}(y,\zeta),\ \forall t>0,$ we get
$$\begin{array}{lcl}
-\frac{1}{n!c_n}\nu_U(x_0,r)&=&\ds\frac{1}{r^{p+1}}\int_{(y,\zeta)\in \{|y-x_0|<r\}\times\rb^n}\left(\frac{r^2}{2}-\frac{n-2}{2n}|y-x_0|^2\right)T(y,\zeta)\wedge\beta^{p}(y,\zeta)\\&+&\ds\frac{1}{r^{p+1}}\int_{(y,\zeta)\in \{|y-x_0|>r\}\times\rb^n}\eta(y)\left(\frac{r^n}{n|y-x_0|^{n-2}}\right)T(y,\zeta)\wedge\beta^{p}(y,\zeta)\\&\leqslant&\ds\frac{r}{2}\nu_T(x_0,r)+\frac{r^{n-p-1}}{n}\int_{(y,\zeta)\in \{|y-x_0|>r_0\}\times\rb^n}\frac{\eta(y)}{|y-x_0|^{n-2}}T(y,\zeta)\wedge\beta^{p}(y,\zeta)\\&+&\ds\frac{r^{n-p-1}}{n}\int_{(y,\zeta)\in \{r\leqslant|y-x_0|<r_0\}\times\rb^n}\frac{1}{|y-x_0|^{n-2}}T(y,\zeta)\wedge\beta^{p}(y,\zeta)\\&=&\ds\frac{r}{2}\nu_T(x_0,r)+\frac{r^{n-p-1}}{n}\int_{(y,\zeta)\in \{|y-x_0|>r_0\}\times\rb^n}\frac{\eta(y)}{|y-x_0|^{n-2}}T(y,\zeta)\wedge\beta^{p}(y,\zeta)\\&+&\ds\frac{r^{n-p-1}}{n}\int_{r}^{r_0}\frac{1}{t^{n-2}}d\sigma_T(x_0,t)\\&=&\ds\frac{r}{2}\nu_T(x_0,r)+\frac{r^{n-p-1}}{n}\int_{(y,\zeta)\in \{|y-x_0|>r_0\}\times\rb^n}\frac{\eta(y)}{|y-x_0|^{n-2}}T(y,\zeta)\wedge\beta^{p}(y,\zeta)\\&+&\ds\frac{r^{n-p-1}}{n}\left(\left[\frac{\sigma_T(x_0,t)}{t^{n-2}}\right]_{r}^{r_0}+(n-2)\int_{r}^{r_0}\frac{\sigma_T(x_0,t)}{t^{n-1}}dt\right)
\end{array}$$
$$\begin{array}{lcl}
&=&\ds\frac{r}{2}\nu_T(x_0,r)+\frac{r^{n-p-1}}{n}\int_{(y,\zeta)\in \{|y-x_0|>r_0\}\times\rb^n}\frac{\eta(y)}{|y-x_0|^{n-2}}T(y,\zeta)\wedge\beta^{p}(y,\zeta)\\&+&\ds\frac{r^{n-p-1}}{n}\left(\frac{\nu_T(x_0,r_{0})}{r_{0}^{n-p-2}}-\frac{\nu_T(x_0,r)}{r^{n-p-2}}+(n-2)\int_{r}^{r_0}\frac{\nu_T(x_0,t)}{t^{n-p-1}}dt\right)
\\&\leqslant&\ds\frac{r}{2}\nu_T(x_0,r)+\frac{r^{n-p-1}}{n}\int_{(y,\zeta)\in \{|y-x_0|>r_0\}\times\rb^n}\frac{\eta(y)}{|y-x_0|^{n-2}}T(y,\zeta)\wedge\beta^{p}(y,\zeta)\\&+&\ds\frac{r^{n-p-1}}{nr_{0}^{n-p-2}}\nu_T(x_0,r_{0})-\frac{r}{n}\nu_T(x_0,r)+\frac{n-2}{n}\int_{r}^{r_0}\nu_T(x_0,t)dt.
\end{array}$$
Since $\nu_T(x_0)=0$, the last inequality leads to $\nu_{U}(x_0)=0$. In the case where $T$ is not smooth, it suffices to use the above discussion for a smooth regularization of $T$ and passing to the limit to obtain our result.
\end{proof}
\section{Continuity of $m$-superHessian operator and $m$-generalized Lelong number}
\subsection{$m$-superHessian operator} Let $\Omega$ be a bounded and strictly convex domain of $\rb^n$. According to Theorem B in \cite{22}, there exists a smooth convex exhaustion function $\psi$ on $\Omega$ such that $dd^{\#}(\psi-|x|^2)\geqslant0$. Motivated by the work of Demailly \cite{3} on the complex Monge-Amp\`ere operator for unbounded plurisubharmonic functions, we are going to define the $m$-superHessian operator $T\wedge\beta^{n-m}\wedge dd^{\#}u_1\wedge...\wedge dd^{\#}u_k$, where $T$ is a closed $m$-positive current of bidimension $(p,p)$ on $\Omega\times\rb^n$ such that $T\wedge\beta^{n-m}$ is positive and the functions $u_1,...,u_k$ are $m$-convex and bounded near $\partial\Omega\cap\Supp T$. For this aim, we need the following version of the Chern-Levine-Nirenberg type inequality in the $m$-superformalism setting.
\begin{pro}[Chern-Levine-Nirenberg type inequality]\label{cln} Assume that $T$ is a closed $m$-positive current of bidimension $(p,p)$ such that $n-p\leqslant m\leqslant n$ and $T\wedge\beta^{n-m}$ is positive, and let $u_1,...,u_k$, $k\leqslant m+p-n$ be locally bounded $m$-convex functions and $v$ is an $m$-convex function locally integrable with respect to $T\wedge\beta^p$. Then, for every compact $K$ and open subset $ L$ with $K\subset { L}\Subset\Omega$, there exists a constant $C_{K, L}\geqslant 0$ such that:
\begin{enumerate}
\item $\|T\wedge\beta^{n-m}\wedge dd^{\#}u_1\wedge...\wedge dd^{\#}u_k\|_K\leqslant C_{K, L}\|u_1\|_{L^\infty( L)}...\|u_k\|_{L^\infty( L)}\|T\wedge\beta^{n-m}\|_{ L}.$
\item $\ds\|vT\wedge\beta^{n-m}\wedge dd^{\#}u_1\wedge...\wedge dd^{\#}u_k\|_K\leqslant C_{K, L}\|u_1\|_{L^\infty( L)}...\|u_k\|_{L^\infty( L)}\int_{ L\times\rb^n}|v|T\wedge\beta^{p}$.
\end{enumerate}
\end{pro}
Note that for the border case $ m = n $, (1) is reduced to Proposition 3. Moreover, in the case of the trivial current $T=1$, the above inequalities are weaker than the one obtained by \cite{15} Lemma 2.1 and by \cite{21} Theorem 3.1 when $u_1=...=u_k=u$.
\begin{proof} (1) Using induction, it suffices to prove that
$$\|T\wedge\beta^{n-m}\wedge dd^{\#}u_1\wedge...\wedge dd^{\#}u_k\|_K\leqslant C_{K, L}\|u_1\|_{L^\infty( L)}\|T\wedge\beta^{n-m}\wedge dd^{\#}u_2\wedge...\wedge dd^{\#}u_k\|_{ L}.$$
In fact, let $\chi$ be a smooth and compactly supported function on $ L$ such that $0\leqslant\chi\leqslant1$ and $\chi\equiv1$ on $K$. Observe that $T\wedge\beta^{n-m}\wedge dd^{\#}u_1\wedge...\wedge dd^{\#}u_k$ is positive because $T\wedge\beta^{n-m}$ is positive. Thus, Proposition 4.1 in \cite{5} combined with an integration by parts, yield
$$\begin{array}{lcl}
\|T\wedge\beta^{n-m}\wedge dd^{\#}u_1\wedge...\wedge dd^{\#}u_k\|_K&\leqslant&\ds C_1\int_{K\times\rb^n}T\wedge\beta^{p-k}\wedge dd^{\#}u_1\wedge...\wedge dd^{\#}u_k\\&\leqslant&\ds C_1\int_{{ L}\times\rb^n}\chi T\wedge\beta^{p-k}\wedge dd^{\#}u_1\wedge...\wedge dd^{\#}u_k\\&=&\ds C_1\int_{{ L}\times\rb^n}u_1T\wedge\beta^{p-k}\wedge dd^{\#}\chi\wedge dd^{\#}u_2\wedge...\wedge dd^{\#}u_k\\&\leqslant&\ds C_2\|u_1\|_{L^\infty( L)}\|T\wedge\beta^{n-m}\wedge dd^{\#}u_2\wedge...\wedge dd^{\#}u_k\|_{ L},
\end{array}$$
where $C_2$ depends only on bounds of coefficients of $dd^{\#}\chi$ and on the set $ L$.

(2) Let us suppose that $v\leqslant0$ on $ L$ and without loss of generality we may assume that $K\subset\Omega_0\subset L$ and $\Omega_0=\{\psi<0\}$, where $\psi$ is a smooth strictly convex function on $\Omega_0$. Assume also that $-1\leqslant u_j\leqslant0$ on $L$ and $u_j=A\psi$ on $\Omega_0\smallsetminus\Omega_\delta$, where $A$ and $\delta>0$ are fixed and $K\subset\Omega_\delta=\{\psi<-\delta\}$. Next, let $\chi$ be a positive, smooth and compactly supported function on $\Omega_0$ such that $\chi=-\psi$ on $\Omega_\delta$. Then, once again Proposition 4.1 in \cite{5}, an integration by parts and (1), yield
$$\begin{array}{lcl}
& &\|vT\wedge\beta^{n-m}\wedge dd^{\#}u_1\wedge...\wedge dd^{\#}u_k\|_K\leqslant\\&\leqslant&C_1\ds\int_{K\times\rb^n}(-v)T\wedge\beta^{p-k}\wedge dd^{\#}u_1\wedge...\wedge dd^{\#}u_k\\&\leqslant&C_1\ds\int_{\Omega_\delta\times\rb^n}(-v)T\wedge\beta^{p-k-1}\wedge dd^{\#}\psi\wedge dd^{\#}u_1\wedge...\wedge dd^{\#}u_k\\&=&C_1\ds\int_{\Omega_\delta\times\rb^n}vT\wedge\beta^{p-k-1}\wedge dd^{\#}\chi\wedge dd^{\#}u_1\wedge...\wedge dd^{\#}u_k\\&=&C_1\left(\ds\int_{\Omega_0\times\rb^n}\chi T\wedge\beta^{p-k-1}\wedge dd^{\#}v\wedge dd^{\#}u_1\wedge...\wedge dd^{\#}u_k-A^k\int_{\Omega_0\smallsetminus\Omega_\delta\times\rb^n}vT\wedge\beta^{p-1}\wedge dd^{\#}\chi\right)\\&\leqslant&C_2\ds\int_{\Omega_0\times\rb^n}T\wedge\beta^{p-1}\wedge dd^{\#}v+C_3\ds\int_{L\times\rb^n}|v|T\wedge\beta^{p}.
\end{array}$$
Using the same line as in the proof of (1), we obtain
$$\int_{\Omega_0\times\rb^n}T\wedge\beta^{p-1}\wedge dd^{\#}v\leqslant C_4\int_{L\times\rb^n}|v|T\wedge\beta^{p},$$
and our proof is concluded for $k\leqslant p-1$. For $k=p$, we can work on $\Omega\times\rb$ instead of $\Omega$ and consider $v,u_1,...,u_p$ as functions on $\Omega\times\rb$ and $T$ as a current on $(\Omega\times\rb)\times(\rb^n\times\rb)$ independent of the extra variables.
\end{proof}
\begin{thm}\label{pr1} Assume that $u$ is $m$-convex and bounded near $\partial\Omega\cap\Supp T$, where $T$ is a closed $m$-positive current of bidimension $(p,p)$ on $\Omega\times\rb^n$ such that $m+p\geqslant n+1$ and $T\wedge\beta^{n-m}$ is positive. Then, the current $uT\wedge\beta^{n-m}$ has locally finite mass on $\Omega\times\rb^n$.
\end{thm}
Theorem \ref{pr1} fails when $m+p=n$ as shown in the example: $\Omega$ is the unit ball of $\rb^n$, $u(x)=-\frac{1}{(\frac{n}{m}-2)|x|^{\frac{n}{m}-2}}$ if $m<\frac{n}{2}$ and $T=(dd^{\#}u)^{m}$.
\begin{proof} We argue as in \cite{3}. Since $\Omega$ is strictly convex, shrinking it a bit, we may assume that $\Omega=\{\psi<0\}$, where $\psi$ is a smooth strictly convex exhaustion function on $\overline{\Omega}$ such that $dd^{\#}\psi\geqslant\beta$ and $\psi=0$, $d\psi\neq0$ on $\partial\Omega$. We set $\Omega_\lambda=\{x\in\Omega;\ \psi(x)<-\lambda\}$ for all $\lambda>0$. By assumption $u$ is bounded near $\partial\Omega\cap\Supp T$, so we choose a sufficiently small constant $\lambda$ and we fix a neighborhood $ L$ of $(\overline{\Omega}\smallsetminus\Omega_\lambda)\cap\Supp T$ such that $\mathscr L(u)\cap L=\emptyset$, where $\mathscr L(u)$ is the unbounded locus of $u$. For $s<0$, we define the function
$$u_s(x)=\left\lbrace\begin{array}{lll}
\max(u(x),A\psi(x))\quad&\mathrm{on}\quad L&\\
\max(u(x),s)\quad&\mathrm{on}\quad\Omega_\lambda.&
\end{array}\right.$$
Since $u$ is bounded on $ L$, by subtracting a constant to $u$, we may assume that $-\mu\geqslant u\geqslant-\delta$ on $ L$ for some constants $\delta>\mu>0$. Then, the definition of $u_s$ on $ L\cup\Omega_\lambda$ (which is a neighborhood of $\overline{\Omega}\cap\Supp T$) is coherent for $A\geqslant\frac{\delta}{\lambda}$ and $s\leqslant-\delta$. Indeed, it suffices to observe that $$\max(u(x),A\psi(x))=\max(u(x),s)=u(x)\quad\mathrm{on}\  L\cap\Omega_\lambda.$$ Thanks to Proposition 2.2 in \cite{17}, the function $u_s$ is also $m$-convex. Now, it is clear that $u_s=A\psi$ on $ L\cap(\Omega\smallsetminus\Omega_{\frac{\mu}{A}})$, hence Stokes formula implies
$$\begin{array}{lcl}
& &\ds\int_{\Omega\times\rb^n}T\wedge\beta^{n-m}\wedge dd^{\#}u_s\wedge(dd^{\#}\psi)^{p+m-n-1}-\ds\int_{\Omega\times\rb^n}T\wedge\beta^{n-m}\wedge dd^{\#}(A\psi)\wedge(dd^{\#}\psi)^{p+m-n-1}\\& &=\ds\int_{\Omega\times\rb^n}dd^{\#}((u_s-A\psi)T\wedge\beta^{n-m}\wedge(dd^{\#}\psi)^{p+m-n-1})=0,
\end{array}$$
because the current $(u_s-A\psi)T\wedge\beta^{n-m}\wedge(dd^{\#}\psi)^{p+m-n-1}$ has a compact support contained in $\overline{\Omega}_{\frac{\mu}{A}}$. Since $u_s$ and $\psi$ both vanish on $\partial\Omega$, Stokes formula gives
$$\begin{array}{lcl}
\ds\int_{\Omega\times\rb^n}u_sT\wedge\beta^{n-m}\wedge(dd^{\#}\psi)^{p+m-n}&=&\ds\int_{\Omega\times\rb^n}\psi T\wedge\beta^{n-m}\wedge dd^{\#}u_s\wedge(dd^{\#}\psi)^{p+m-n-1}\\&\geqslant&-\|\psi\|_{L^\infty(\Omega)}\ds\int_{\Omega\times\rb^n}T\wedge\beta^{n-m}\wedge dd^{\#}u_s\wedge(dd^{\#}\psi)^{p+m-n-1}\\&=&-A\|\psi\|_{L^\infty(\Omega)}\ds\int_{\Omega\times\rb^n}T\wedge\beta^{n-m}\wedge(dd^{\#}\psi)^{p+m-n}.
\end{array}$$
Finally, take $A=\frac{\delta}{\lambda}$, let $s$ tend to $-\infty$ and use the inequality $u\geqslant-\delta$ on $ L$, to obtain
$$\begin{array}{lcl}
\ds\int_{\Omega\times\rb^n}uT\wedge\beta^{n-m}\wedge(dd^{\#}\psi)^{p+m-n}&\geqslant&-\delta\ds\int_{ L\times\rb^n}T\wedge\beta^{n-m}\wedge(dd^{\#}\psi)^{p+m-n}\\&+&\ds\lim_{s\rightarrow-\infty}\int_{\Omega_\lambda\times\rb^n}u_sT\wedge\beta^{n-m}\wedge(dd^{\#}\psi)^{p+m-n}\\&\geqslant&-(\delta+\|\psi\|_{L^\infty(\Omega)}\frac{\delta}{\lambda})\ds\int_{\Omega\times\rb^n}T\wedge\beta^{n-m}\wedge(dd^{\#}\psi)^{p+m-n}.
\end{array}$$
Since $T\wedge\beta^{n-m}$ is positive, Proposition 3.2 in \cite{5} implies that the last integral is finite. Therefore, $uT\wedge\beta^{n-m}$ has locally finite mass.
\end{proof}
\begin{lem}\label{lm1} Let $(u_k)$ be a decreasing sequence of $m$-convex functions converging pointwise to  $u$ on $\Omega$ and $(\nu_k)$ is a sequence of positive measure converging weakly to $\nu$ on $\Omega$. Then every weak limit $\mu$ of $(u_k\nu_k)$ satisfies $\mu\leqslant u\nu$.
\end{lem}
The proof of Lemma \ref{lm1} is matching the one given by Demailly in \cite{3}. Assume that $u_1,...,u_k$ are $m$-convex on $\Omega$ such that each $u_j$ is bounded near $\partial\Omega\cap\Supp T$, where $T$ is a closed $m$-positive current of bidimension $(p,p)$ on $\Omega\times\rb^n$ such that $m+p\geqslant n+1$ and $T\wedge\beta^{n-m}$ is positive. As an immediate consequence of Theorem \ref{pr1} and similarly to the complex Hessian theory made by Dhouib and Elkhadhra \cite{16}, we define by induction the following positive current:
$$T\wedge\beta^{n-m}\wedge dd^{\#}u_1\wedge...\wedge dd^{\#}u_k=dd^{\#}(u_1T\wedge\beta^{n-m}\wedge dd^{\#}u_2\wedge...\wedge dd^{\#}u_k).$$
In fact, it suffices to explain the case $k=1$. Setting $T\wedge\beta^{n-m}\wedge dd^{\#}u_1=dd^{\#}(u_1T\wedge\beta^{n-m})$, we get a positive current. To see the positivity we consider a sequence $(u_1^j)_j$ of smooth and $m$-convex functions that decreases to $u_1$ and observe by the Lebesgue's bounded convergence theorem that the positive sequence of currents $T\wedge\beta^{n-m}\wedge dd^{\#}u_1^j$ converges weakly to $T\wedge\beta^{n-m}\wedge dd^{\#}u_1$. Next, according to Lagerberg \cite{5}, we point out that there is a significant difference between the complex theory and the context of supercurrent. Indeed weakly positive current do not satisfies a trace measure inequality as obtained by Demailly \cite{3} in the complex setting. Hence, the continuity of the $m$-superHessian operator relatively to a given $m$-positive closed current $T$ and under decreasing sequence of $m$-convex functions requires the positivity of the current $T\wedge\beta^{n-m}$. More precisely, we obtain the following main superformalism version:
\begin{thm C} Let $\Omega$ be a bounded and strictly convex domain of $\rb^n$. Assume that $u_0,...,u_k$ are $m$-convex on $\Omega$ such that each $u_s$ is bounded near $\partial\Omega\cap\Supp T$, where $T$ is a closed $m$-positive current of bidimension $(p,p)$ on $\Omega\times\rb^n$ such that $1\leqslant k\leqslant m+p-n$ and $T\wedge\beta^{n-m}$ is positive. Then, if $u^{j}_0,...,u^{j}_k$ are decreasing sequences of $m$-convex functions which converging pointwise and respectively to $u_0,...,u_k$, in the sense of currents, we have:
\begin{enumerate}
\item $u_{0}^{j}T\wedge\beta^{n-m}\wedge dd^{\#}u_{1}^{j}\wedge...\wedge dd^{\#}u_{k}^{j}\longrightarrow u_{0}T\wedge\beta^{n-m}\wedge dd^{\#}u_{1}\wedge...\wedge dd^{\#}u_{k}$, for $k\neq m+p-n$.
\item $T\wedge\beta^{n-m}\wedge dd^{\#}u_{1}^{j}\wedge...\wedge dd^{\#}u_{k}^{j}\longrightarrow T\wedge\beta^{n-m}\wedge dd^{\#}u_{1}\wedge...\wedge dd^{\#}u_{k}$.
\end{enumerate}
\end{thm C}
Notice that Theorem C is the corresponding result in the superformalism setting of Theorem 2 in \cite{16} for the complex hessian theory. On the other hand, it should be noted that when $T=1$ and $k=m$ the second statement of Theorem C was proved by \cite{15} under the weaker hypothesis: for $s=1,...,k$, the sequence $(u^j_s)$ converges weakly to $u_s$. Comparing with Theorem 2.6 in \cite{15}, we see that the first statement of Theorem C was proved for the particular case $T=1$ and $k=m$ but with the hypothesis $u_0$ is locally bounded instead of each $u_s$ is bounded near the boundary.  Moreover, statement (1) fails when $k=m+p-n$ as shown by the following example: $T=1$ and $u_0(x)=...=u_k(x)=-\frac{1}{(\frac{n}{m}-2)|x|^{\frac{n}{m}-2}}$ if $m<\frac{n}{2}$.
\begin{proof} The proof follows the same techniques used by Demailly \cite{3} for plurisubharmonic functions. Observing that $(1)\Rightarrow(2)$, then it suffices to prove $(1)$. Since $T$ is a closed $m$-positive current, $(1)$ is obvious for $k=1$ thanks to Theorem \ref{pr1} and Lebesgue's bounded convergence theorem. Since the sequence $(u_{i}^{j})_j$ is decreasing and since $u_i$ is bounded near $\partial\Omega\cap\Supp T$, the family $(u_{i}^{j})_{j\in\nb}$ is uniformly bounded near $\partial\Omega\cap\Supp T$. Without loss of generality, we can assume that $\Omega=\{\psi<0\}$,  where $\psi$ is a smooth strictly convex exhaustion function on $\overline{\Omega}$ such that $dd^{\#}\psi\geqslant\beta$ and $\psi=0$, $d\psi\neq0$ on $\partial\Omega$. We set $\Omega_\lambda=\{\psi<-\lambda\}$, for all $\lambda>0$. After addition of a constant we can assume that $u_{i}^j\leqslant-1$ on $\overline{\Omega}$. We fix $\lambda$ so small such that $u_{i}^j$ is bounded near $(\overline{\Omega}\smallsetminus\Omega_\lambda)\cap\Supp T$ and we select a neighborhood $ L$ of $(\overline{\Omega}\smallsetminus\Omega_\lambda)\cap\Supp T$ such that $u_{i}^j$ is bounded near $\overline{ L}$. Let
$$v_{i}^j(x)=\left\lbrace\begin{array}{lll}
\max(u_{i}^j(x),A\psi(x))\quad&\mathrm{on}\quad L&\\
u_{i}^j(x)\quad&\mathrm{on}\quad\overline{\Omega}_\lambda\cap\Supp T&
\end{array}\right.,\ \forall i=0,...,k.$$
By construction $u_{i}^j\geqslant-M$ on $ L$ for some constant $M>0$. Set $A=\frac{M}{\lambda}$, then $v_{i}^j$ coincides with $u_{i}^j$ on $ L\cap\Omega_\lambda$ since $A\psi\leqslant-A\lambda=-M$ on $\overline{\Omega}_\lambda\cap\Supp T$ and $v_{i}^j=A\psi$ on $ L\cap(\Omega\smallsetminus\Omega_{\frac{\lambda}{M}})$. It follows by Proposition 2.2 in \cite{17} that each $v_i^j$ is $m$-convex. Denote
by $(u_{i}^{j,\varepsilon})_{0<\varepsilon\leqslant\varepsilon_0}$ an increasing family of regularization functions converging to $u_{i}^j$ as $\varepsilon\rightarrow0$ such that $-M\leqslant u_{i}^{j,\varepsilon}\leqslant-1$ on $ L$. Let
$$v_{i}^{j,\varepsilon}(x)=\left\lbrace\begin{array}{lll}
\max_\varepsilon(u_{i}^{j,\varepsilon}(x),A\psi(x))\quad&\mathrm{on}\quad L&\\
u_{i}^{j,\varepsilon}(x)\quad&\mathrm{on}\quad\overline{\Omega}_\lambda\cap\Supp T&
\end{array}\right.,\ \forall i=0,...,k,$$
where $\max_\varepsilon=\max\ast\rho_\varepsilon$ is a regularized max function. By considering the sequences $v_i^j$ (respectively $v_{i}^{j,\varepsilon}$), we may assume that all $u_{i}^j$ (and similarly all $v_{i}^{j,\varepsilon}$) coincide with $A\psi$ on a fixed neighborhood of $\partial\Omega\cap\Supp T$. Assume that $(1)$ has been proved for $k$, then $$S^j:=T\wedge\beta^{n-m}\wedge dd^{\#}u_{1}^{j}\wedge...\wedge dd^{\#}u_{k}^{j}\longrightarrow S:=T\wedge\beta^{n-m}\wedge dd^{\#}u_{1}\wedge...\wedge dd^{\#}u_{k}.$$
Since $T\wedge dd^{\#}u^{j}_{1}\wedge...\wedge dd^{\#}u^{j}_{k}$ is a closed $m$-positive current and $S^{j}$ is positive, proof of Theorem \ref{pr1} implies that the sequence $(u_{1}^jS^{j})_j$ has locally uniformly bounded mass, hence is relatively compact for the weak topology. In order to prove $(1)$, we only have to show that every weak limit $\Theta$ of $u_{0}^jS^j$ is equal to $u_{0}S$. Let $(q,q)$ be the bidimension of $S$ and let $\alpha$ be an arbitrary smooth compactly supported form of bidegree $(q,0)$. Then, the positive measure $S^j\wedge\sigma_{q}\alpha\wedge J(\alpha)$ converges weakly to $S\wedge\sigma_{q}\alpha\wedge J(\alpha)$ and by Lemma \ref{lm1}, we can see that $\Theta\wedge\sigma_{q}\alpha\wedge J(\alpha)\leqslant u_0S\wedge\sigma_{q}\alpha\wedge J(\alpha)$, which means that $u_0S-\Theta$ is positive. Next, we need to show the inequality
$$\begin{array}{lcl}
& &\ds\int_{\Omega\times\rb^n}u_0T\wedge\beta^{n-m}\wedge dd^{\#}u_{1}\wedge...\wedge dd^{\#}u_{k}\wedge(dd^{\#}\psi)^q\leqslant\\&\leqslant&\ds\liminf_{j\rightarrow+\infty}\int_{\Omega\times\rb^n}u_{0}^{j}T\wedge\beta^{n-m}\wedge dd^{\#}u_{1}^{j}\wedge...\wedge dd^{\#}u_{k}^{j}\wedge(dd^{\#}\psi)^q.
\end{array}$$
As $u_0\leqslant u_{0}^{j}\leqslant u_{0}^{j,\varepsilon_0},\ \forall\varepsilon_0>0$, an integration by parts yields
$$\begin{array}{lcl}
& &\ds\int_{\Omega\times\rb^n}u_0T\wedge\beta^{n-m}\wedge dd^{\#}u_{1}\wedge...\wedge dd^{\#}u_{k}\wedge(dd^{\#}\psi)^q\leqslant\\&\leqslant&\ds\int_{\Omega\times\rb^n}u_{0}^{j,\varepsilon_0}T\wedge\beta^{n-m}\wedge dd^{\#}u_{1}\wedge...\wedge dd^{\#}u_{k}\wedge(dd^{\#}\psi)^q\\&=&\ds\int_{\Omega\times\rb^n}u_{1}T\wedge\beta^{n-m}\wedge dd^{\#}u_{0}^{j,\varepsilon_0}\wedge dd^{\#}u_{2}\wedge...\wedge dd^{\#}u_{k}\wedge(dd^{\#}\psi)^q
\end{array}$$
By repeating this argument with $u_1,...,u_k$, we obtain
$$\begin{array}{lcl}
& &\ds\int_{\Omega\times\rb^n}u_0T\wedge\beta^{n-m}\wedge dd^{\#}u_{1}\wedge...\wedge dd^{\#}u_{k}\wedge(dd^{\#}\psi)^q\leqslant\\&\leqslant&\ds\int_{\Omega\times\rb^n}u_{k}T\wedge\beta^{n-m}\wedge dd^{\#}u_{0}^{j,\varepsilon_0}\wedge...\wedge dd^{\#}u_{k-1}^{j,\varepsilon_{k-1}}\wedge(dd^{\#}\psi)^q\\&\leqslant&\ds\int_{\Omega\times\rb^n}u_{0}^{j,\varepsilon_0}T\wedge\beta^{n-m}\wedge dd^{\#}u_{1}^{j,\varepsilon_1}\wedge...\wedge dd^{\#}u_{k}^{j,\varepsilon_k}\wedge(dd^{\#}\psi)^q.
\end{array}$$
Now, we let $\varepsilon_k\rightarrow0,...,\varepsilon_0\rightarrow0$ in this order, so we have weak convergence at each step and $u_{0}^{j,\varepsilon_0}=0$ on the boundary. Hence, the last integral converges and we get the desired inequality. Therefore, $(u_0S-\Theta)\wedge(dd^{\#}\psi)^q=0$ and since $u_0S-\Theta$ is positive, then thanks to Lagerberg \cite{5} we obtain the equality $u_0S=\Theta$.
\end{proof}
\subsection{$m$-generalized Lelong numbers} Let us first define the pseudo-sphere and the pseudo-ball associated with an $m$-convex function $\varphi$, respectively by
$$S(r)=\{x\in\Omega;\ \varphi(x)=r\}\qquad\mathrm{and}\qquad B(r)=\{x\in\Omega;\ \varphi(x)<r\}.$$
Thanks to the argument before Theorem C, the $m$-superHessian operator $T\wedge\beta^{n-m}\wedge (dd^{\#} \varphi)^{m+p-n}$ is a well defined positive measure provided that $T$ is a closed $m$-positive current of bidimension $(p,p)$ on $\Omega\times\rb^n$ such that $n-p\leqslant m\leqslant n$, $T\wedge\beta^{n-m}$ is positive and $\varphi$ is $m$-convex bounded near $\partial\Omega\cap\Supp T$. Hence, inspired by the work of Elkhadhra \cite{20}, we state the following definition:
\begin{defn}\label{df1} Assume that $1\leqslant m\leqslant\frac{n}{2}$ and $\varphi$ is semi-exhaustive on $\Supp T$, i.e. there exists $R$ such that $B(R)\cap\Supp T\Subset\Omega$, then $S(-\infty)\cap\Supp T$ is compact and therefore for $r<R$, we set
$$\nu_{T}^m(\varphi,r)=\int_{B(r)\times\rb^n}T\wedge\beta^{n-m}\wedge(dd^{\#} \varphi)^{m+p-n}\ {\rm and}\ \nu_{T}^m(\varphi)=\lim_{r\rightarrow-\infty}\nu_{T}^m(\varphi,r).$$
The number $\nu_{T}^m(\varphi)$ will be called the $m$-generalized Lelong number with respect to the weight $\varphi$.
\end{defn}
Observe that $\nu_{T}^m(\varphi)$ measure the asymptotic behaviour of the current $T\wedge\beta^{n-m}\wedge(dd^{\#}\varphi)^{m+p-n}$ near the $m$-polar set $\{\varphi=-\infty\}$. Moreover, this notion is a real Hessian version of the definition of $m$-generalized Lelong number given by Elkhadhra \cite{20} in the complex Hessian setting.
\begin{pro}
For any convex increasing function $\chi: \rb\rightarrow\rb$, we have
\begin{equation}\label{eq01}
\int_{B(r)\times\rb^n}T\wedge\beta^{n-m}\wedge(dd^{\#}\chi\circ\varphi)^{m+p-n}=\left(\chi'(r-0)\right)^{m+p-n}\nu_{T}^m(\varphi,r),
\end{equation}
where $\chi'(r-0)$ is the left derivative of $\chi$ at $r$.
\end{pro}
\begin{proof} According to Trudinger and Wang \cite{21}, the composit function $\chi\circ\varphi$ is $m$-convex. Now, by adapting the same techniques of Demailly \cite{3}, let $\chi_\varepsilon$ be the convex function equal to $\chi$ on $[r-\varepsilon,+\infty[$ and to a linear function of slope $\chi'(r-\varepsilon-0)$ on $[-\infty,r-\varepsilon]$. So, we get $dd^{\#}(\chi_\varepsilon\circ\varphi)=\chi'(r-\varepsilon-0)dd^{\#}\varphi$ on $B(r-\varepsilon)$
and Stokes formula implies
$$\begin{array}{lcl}
\ds\int_{B(r)\times\rb^n}T\wedge\beta^{n-m}\wedge(dd^{\#}\chi\circ\varphi)^{m+p-n}&=&\ds\int_{B(r)\times\rb^n}T\wedge\beta^{n-m}\wedge(dd^{\#}\chi_\varepsilon\circ\varphi)^{m+p-n}\\&\geqslant&\ds\int_{B(r-\varepsilon)\times\rb^n}T\wedge\beta^{n-m}\wedge(dd^{\#}\chi_\varepsilon\circ\varphi)^{m+p-n}\\&=&\ds\chi'(r-0)^{m+p-n}\nu_{T}^m(\varphi,r-\varepsilon).
\end{array}$$
Similarly, taking $\widetilde{\chi}_\varepsilon$ equal to $\chi$ on $[-\infty,r-\varepsilon]$ and linear on $[r-\varepsilon,r]$, we obtain
$$\begin{array}{lcl}
\ds\int_{B(r-\varepsilon)\times\rb^n}T\wedge\beta^{n-m}\wedge(dd^{\#}\chi\circ\varphi)^{m+p-n}&\leqslant&\ds\int_{B(r)\times\rb^n}T\wedge\beta^{n-m}\wedge(dd^{\#}\widetilde{\chi}_\varepsilon\circ\varphi)^{m+p-n}\\&=&\ds\chi'(r-0)^{m+p-n}\nu_{T}^m(\varphi,r).
\end{array}$$
The desired formula follows when $\varepsilon$ tends to $0$.
\end{proof}
\begin{pcs} Assume that $a\in\Omega$ and let $\varphi_m(x) =-\frac{1}{(\frac{n}{m}-2)|x-a|^{\frac{n}{m}-2}}$ if $m\neq\frac{n}{2}$ and $\log|x-a|$ if $m=\frac{n}{2}$. A straightforward computation gives
$$dd^{\#}\varphi_m(x)=\left\lbrace\begin{array}{lll}
\ds|x-a|^{-\frac{n}{m}}\left(\beta-\frac{n}{m}|x-a|^{-2}\frac{1}{2}d|x-a|^2\wedge\frac{1}{2}d^{\#}|x-a|^2\right)\ &\mathrm{if}\quad m\neq\frac{n}{2}&\\
|x-a|^{-2}\left(\beta-2|x-a|^{-2}\frac{1}{2}d|x-a|^2\wedge\frac{1}{2}d^{\#}|x-a|^2\right)\quad&\mathrm{if}\quad m=\frac{n}{2}.&
\end{array}\right.$$
By using the Bin\^ome formula, for $x\not=a$, a simple computation gives
$$\forall s<m,\qquad (dd^{\#}\varphi_m(x))^s\wedge\beta^{n-s}=\left\lbrace\begin{array}{lll}
\left(1-\frac{s}{m}\right)|x-a|^{-\frac{ns}{m}}\beta^{n}\quad&\mathrm{if}\quad m\neq\frac{n}{2}&\\
|x-a|^{-2s}(1-\frac{2s}{n})\beta^n&\mathrm{if}\quad m=\frac{n}{2}.&
\end{array}\right.$$
In particular, we see that $\varphi_m$ is $m$-convex and we have the following two important cases:

(1) For $1\leqslant m<\frac{n}{2}$, put $\mu=1-\frac{n}{2m}$, then by applying formula (\ref{eq01}) for the convex increasing function $\chi$ defined by $\chi(x)=(-x)^{\frac{1}{\mu}}$ if $x\leqslant R$ and $\chi$ is linear on $\{x\geqslant R\}$. Then, for $r<R$, we obtain
$$\int_{B(r)\times\rb^n}T\wedge\beta^{n-m}\wedge\left(dd^{\#}(-\varphi_m)^{\frac{1}{\mu}}\right)^{m+p-n}=\left(\frac{2m}{n-2m}\right)^{m+p-n}(-r)^{\frac{n(m+p-n)}{2m-n}}\nu_{T}^m(\varphi_m,r).$$
For $t>0$, let $r=-\frac{1}{\frac{n}{m}-2}t^{2\mu}$, we have
$$\nu_{T}^m(\varphi_m,r)=\nu_{T}^m\left(\varphi_m,-\frac{1}{\frac{n}{m}-2}t^{2\mu}
\right)=\frac{1}{t^{\frac{n(m+p-n)}{m}}}\ds\int_{\{|x-a|<t\}\times\rb^n}T\wedge\beta^{p}=\frac{p!}{t^{\frac{n(m+p-n)}{m}}}\Theta_T(a,t),$$
here $\Theta_T(a,t)$ is the trace measure of $T$ on the euclidean ball $\mathbb{B}(a,t)$ (see \cite{5}).
Now, for $m=\frac{n}{2}$, we apply formula (\ref{eq01}) for the convex increasing function $\chi$ defined by $\chi(x)=\exp\left(2x\right)$. Then, for $r<R$, we obtain
$$\int_{B(r)\times\rb^n}\left(dd^{\#}\exp\left(2\varphi_m\right)\right)^{m+p-n}\wedge T\wedge\beta^{n-m}=2^{m+p-n}\exp(2r(m+p-n))\nu_{T}^m(\varphi_m,r).$$
For $t>0$, let $r=\log{t}$, we have
$$\nu_{T}^m(\varphi_m,r)=\nu_{T}^m\left(\varphi_m,\log t
\right)=\frac{1}{t^{2(p-\frac{n}{2})}}\ds\int_{\{|x-a|<t\}\times\rb^n}T\wedge\beta^{p}=\frac{p!}{t^{2(p-\frac{n}{2})}}\Theta_T(a,t).$$
Thus, the number
$$\nu_{T}^m(a):=\ds\lim_{t\l 0^+}\frac{p!}{t^{\frac{n(m+p-n)}{m}}}\Theta_T(a,t)$$ exists and it coincides with the $m$-Lelong number of $T$ at $a$ as established by Elkhadhra and Zahmoul \cite{18} in the case where $T$ is $m$-positive and $T\wedge\beta^{p-1}$ is convex.

(2) For $\frac{n}{2}<m\leqslant n$ and according to Trudinger and wang \cite{19}, every $m$-convex function satisfies a local H\"older estimate, and therefore its $m$-polar set is empty. However, in this case and by an adaptation of the proof of Proposition 4.1 in \cite{18}, if we assume that $\varphi$ is a positive $m$-convex function such that $\varphi^\mu$ is $m$-convex, then we have
$$\begin{array}{lcl}
\ds\int_{B(r_1,r_2)\times\mathbb{R}^n}T\wedge\beta^{n-m}\wedge(dd^{\#}{\varphi^\mu})^{m+p-n}&=&\ds\frac{\mu^{m+p-n}}{r_2^{\frac{n}{2m}(m+p-n)}}\int_{B(r_2)\times\mathbb{R}^n}T\wedge\beta^{n-m}\wedge(dd^{\#}\varphi)^{m+p-n}
\\&-&\ds\frac{\mu^{m+p-n}}{r_1^{\frac{n}{2m}(m+p-n)}}\int_{B(r_1)\times\mathbb{R}^n}T\wedge\beta^{n-m}\wedge(dd^{\#}\varphi)^{m+p-n}.
\end{array}$$
In particular, we see that the map $$r\longmapsto \nu_T^m(\varphi,r):=\frac{\mu^{m+p-n}}{r^{\frac{n}{2m}(m+p-n)}}\int_{B(r)\times\mathbb{R}^n}T\wedge\beta^{n-m}\wedge(dd^{\#}\varphi)^{m+p-n}$$
is positive and increases, and therefore we can define the $m$-generalized Lelong number of $T$ relatively to the weight $\varphi$ by setting
$\nu_T^m(\varphi)=\ds\lim_{r\l 0}\nu_T^m(\varphi,r)$. In particular, for $m=n$, $$\nu_T(\varphi):=\lim_{r\rightarrow0}\frac{1}{2^pr^\frac{p}{2}}\int_{B(r)\times\mathbb{R}^n}T\wedge(dd^{\#} \varphi)^{p}$$ exists and it is the Lelong number of $T$ relative to the weight $\varphi$ as defined by \cite{18}. Moreover, if we  assume that $\varphi=|x-a|^2$ and $r=t^2$, then the $m$-superHessian operator $T\wedge\beta^{p}$ is well defined without the condition $T\wedge\beta^{n-m}$ is positive. Hence, as an immediate consequence of the above discussion, the function $$t\longmapsto \frac{p!}{t^{\frac{n(m+p-n)}{m}}}\Theta_T(a,t)$$ is positive and increasing on $]0,+\infty[$. Therefore, the $m$-Lelong number
$\nu_{T}^m(a)$ of $T$ at $a$ exists. In particular, for $m=n$, we recover the Lelong number defined by Lagerberg \cite{5}.
\end{pcs}
\begin{rem}\
\begin{enumerate}
\item As a direct applications of the monotonicity statement and by an adaptation of the proof of Proposition 5.11 in \cite{3}, we see that for each $\delta>0$, the set $\mathscr E_\delta=\{z\in\Omega:\ \nu_T^m(z)\geqslant \delta\}$ is closed and has a locally finite $\mathscr H_{\frac{n(m+p-n)}{m}}$ Hausdorff measure in $\Omega$. Also, by the same lines of the proof of Proposition 4.2 in \cite{18}, we gets that for each compact subset $K$ of $\rb^n$, if ${\mathscr H}_{\frac{n(m+p-n)}{m}}(K\cap\Supp T)=0$, then $\|T\|_K=0$. Finally, for $m=n-p$, it is clear that the $(n-p)$-generalized Lelong number of $T$ at the point $a$ is exactly the trace measure of $T$ on $\{a\}$.
\item As a special case when $T=dd^{\#}u$ and $0<m\leqslant\frac{n}{2}$, where $u$ is an $m$-convex function, we recover the definition of the generalized Lelong number of $u$ with respect to the weight $\varphi$ given by Wan and Wang \cite{17} (modulo a constant).
\item If $T$ is $m$-positive and $(m-1)$-positive (for example when $T=(dd^{\#}u)^p$, $u$ is $m$-convex and bounded near $\partial\Omega$), then $\nu_T^{m-1}(a)=0,\ \forall a\in\Omega$. In particular, if $T$ is a strongly positive closed current, then $T$ is $m$-positive for every $p\leqslant m\leqslant n$, and therefore $\nu_T^{j}(a)=0,\ \forall j=p,...,n-1$ and $a\in\Omega$. Moreover, if $u$ is $m$-convex, then $\nu_{dd^{\#}u}^j(a)=0,\ \forall j=2,...,m-1$ and $a\in\Omega$.
\item Assume that $\Omega=\rb^n$ and $T=1$. Since $p=n$, By Example 4.1 in \cite{5} which gives $\Theta_1(a,t)=t^n Vol_n(\mathbb{B}(a,1))$ and by the above computations, for all $r>0$ and $1\leqslant m\leqslant\frac{n}{2}$, we have
$$\int_{\mathbb{B}(a,r)}(dd^{\#}\varphi_m)^m\wedge\beta^{n-m}=n!Vol_n(\mathbb{B}(a,1)).$$
In particular, we derive the following fundamental formula which is the corresponding of a well-known formula in the complex Hessian theory:
$$(dd^{\#}\varphi_m)^m\wedge\beta^{n-m}=n!Vol_n(\mathbb{B}(a,1))\delta_a,\quad\forall1\leqslant m\leqslant\frac{n}{2}.$$
\end{enumerate}
\end{rem}
Next, we state an analogue of the first Demailly comparison theorem in our setting which is the counterpart of Theorem 1 in \cite{20}. We omit the proof since it is almost identical to the one given by Demailly \cite{3}.
\begin{thm} [Demailly's comparison theorem] Assume that $T$ is a closed $m$-positive current of bidimension $(p,p)$ on $\Omega\times\rb^n$ such that $T\wedge\beta^{n-m}$ is positive. Let $\varphi,\psi\ :\ \Omega\rightarrow[-\infty,+\infty[$ be two continuous $m$-convex functions. We assume that $\varphi,\psi$ are semi-exhaustive on $\Supp T$ and that
$$l:=\limsup{{\psi(x)}\over{\varphi(x)}}<+\infty\ \ {\rm as}\ \ x\in\Supp T\ \ {\rm and}\ \ \varphi(x)\rightarrow -\infty.$$ Then, $\nu_{T}^m(\psi)\leq l^{m+p-n}\nu_{T}^m(\varphi)$, and the equality holds if $l=\lim\frac{\psi}{\varphi}$.
\end{thm}
As a consequence of the comparison theorem, we see that the $m$-generalized Lelong number $\nu_{T}^m(\varphi)$ only depends on the asymptotic behavior of $\varphi$ near $\{\varphi=-\infty\}\cap\Supp T$.
Let $f:\rb^n\rightarrow\rb^m$ be an affine function. Following Lagerberg \cite{5}, the function $f$ can be extended to unique affine application $\tilde{f}:\rb^n\times\rb^n\rightarrow\rb^m\times\rb^m$ such that $\tilde{f}\circ J=J\circ\tilde{f}$. If $T$ is a current of bidimension $(p,p)$ on $\rb^n\times\rb^n$ and $f$ is proper on $\Supp T$, then the direct image of $T$ by $\tilde f$ is a current of bidimension $(p,p)$ on $\rb^m\times\rb^m$ noted $f_\ast T$ and defined by
\begin{equation}\label{e3}
\langle f_\ast T,\alpha\rangle=\langle T,f^\ast\alpha\rangle,
\end{equation}
for every compactly supported form $\alpha$ of bidegree $(p,p)$ on $\rb^m\times\rb^m$. The definition makes sense, because $\Supp T\cap f^{-1}(\Supp\alpha)$ is compact. Moreover, thanks to Lagerberg \cite{5}, if $T$ is (weakly) positive, then $f_\ast T$ is also (weakly) positive. Our aim now is to find a relationship between the Lelong number of an $m$-positive current and the Lelong number of his direct image by a projection. Strongly inspired by the work of Demailly \cite{3} on the same subject in the complex setting, we have proven the following result.
\begin{pro} Let $\pi$ be the projection of $\rb^n$ on $\rb^{n-k}$ and $T$ be an $m$-positive closed current of bidimension $(p,p)$ on $\rb^n\times\rb^n$ such that $T\wedge\beta^{n-m}$ is positive, $1\leqslant k\leqslant\inf(m-1,n-m)$ and $\pi_\ast T$ is $(m-k)$-positive closed on $\rb^{n-k}\times\rb^{n-k}$. Let $\psi$ be an $m$-convex function on $\rb^{n-k}$ which is semi-exhaustive on $\Supp \pi_\ast T$, i.e. there exists $R$ such that $B(R)\cap \Supp\pi_\ast T\Subset\Omega$. Then, $\varphi=\psi\circ\pi$ is $m$-convex on $\rb^n$ and semi-exhaustive on $\Supp T$, and we have
$$\int_{\lbrace\varphi<r\rbrace\times\rb^{n}}T\wedge(dd^{\#}\varphi)^{m+p-n}\wedge\beta^{n-m}=\int_{\lbrace\psi<r\rbrace\times\rb^{n-k}}\pi_\ast T\wedge(dd^{\#}\psi)^{(m-k)+p-(n-k)}\wedge\beta'^{(n-k)-(m-k)},$$
for all $r<R$, where $\beta'=\frac{1}{2}dd^{\#}|\pi(x)|^2,\ \forall x\in\rb^n$. In particular, we have $\nu^m_{\pi_\ast T}(\psi)=\nu^{m-k}_T(\psi\circ\pi)$.
\end{pro}
\begin{proof} If $\alpha$ is an $m$-positive form of bidegree $(1,1)$ on $\rb^{n-k}$, then $\pi^\ast\alpha$ is also an $m$-positive form on $\rb^{n}$. Indeed, for $s\leqslant m$,
$$(\pi^\ast\alpha)^s\wedge\beta^{n-s}=(\pi^\ast\alpha)^s\wedge(\beta'+\beta'')^{n-s}=(\pi^\ast\alpha)^s\wedge\beta'^{n-k-s}\wedge\beta''^k\geqslant 0.$$
Assume first that $\psi$ is smooth. In order to prove that $\varphi=\psi\circ\pi$ is $m$-convex, it suffices to prove that $dd^{\#}\varphi$ is $m$-positive. Since $dd^{\#}\psi$ is $m$-positive and since $\pi^\ast$ commute with $d$ and $d^{\#}$ (see \cite{5}), then $dd^{\#}\varphi=dd^{\#}(\psi\circ\pi)=dd^{\#}(\pi^\ast\psi)=\pi^\ast(dd^{\#}\psi)$ is $m$-positive thanks to the above discussion.
For $r<R$, we have
$$\begin{array}{lcl}
& &\ds\int_{\lbrace\psi<r\rbrace\times\rb^{n-k}}\pi_\ast T\wedge(dd^{\#}\psi)^{(m-k)+p-(n-k)}\wedge\beta'^{(n-k)-(m-k)}=\\&=&\ds\int_{\rb^{n-k}\times\rb^{n-k}}\pi_\ast T\wedge{\1}_{\lbrace\psi<r\rbrace}(dd^{\#}\psi)^{m+p-n}\wedge\beta'^{n-m}\\&=&\ds\int_{\rb^n\times\rb^n}T\wedge\pi^\ast({\1}_{\lbrace\psi<r\rbrace}(dd^{\#}\psi)^{m+p-n}\wedge\beta'^{n-m})\\&=&\ds\int_{\rb^{n}\times\rb^{n}}T\wedge({\1}_{\lbrace\psi<r\rbrace}\circ\pi)(dd^{\#}\pi^\ast(\psi))^{m+p-n}\wedge\beta^{n-m}\\&=&\ds\int_{\rb^{n}\times\rb^{n}}T\wedge{\1}_{\lbrace\psi\circ\pi<r\rbrace}(dd^{\#}(\psi\circ\pi))^{m+p-n}\wedge\beta^{n-m}\\&=&\ds\int_{\lbrace\varphi<r\rbrace\times\rb^{n}}T\wedge(dd^{\#}\varphi)^{m+p-n}\wedge\beta^{n-m}.
\end{array}$$
This follows almost immediately from the equality (\ref{e3}) when $\psi$ is smooth and when we write ${\1}_{\lbrace\psi<r\rbrace}$ as the limit of an increasing sequence of smooth functions. In
general, if $\psi$ is not necessarily smooth, we take a smooth regularization $(\psi_j)_j$ of $\psi$. Then, the equality is hold for $\varphi_j$ and $\psi_j$. Thanks to Theorem C, the associated superHessian operators converge to measures which can be considered with compact supports because our functions are semi-exhaustive. So, by passing to the limit it is not difficult to deduce the equality for $\varphi$ and $\psi$. The later statement of Proposition 6 is obtained by sending $r\rightarrow-\infty$.
\end{proof}
\section{Weighted relative extremal function and weighted $m$-Hessian capacity}
Let $\Omega$ be an open bounded subset of $\rb^n$ and denote by $\mathscr{C}^{-}_m(\Omega)$ the set of all negative $m$-convex functions on $\Omega$. In this section we begin by introducing analogously to the complex Hessian setting (see Lu \cite{27}) two classes of Cegrell-type $\mathcal{E}_{m}^0$ and $\mathcal{F}_m$.  After investigating some properties of these classes and by using ideas of Nguyen \cite{25} from pluripotential theory in $\cb^n$, we introduce the weighted extremal function as well as the weighted $m$-Hessian capacity in the real Hessian context. The aim is to generalize the notions of $m$-Hessian capacity and extremal function studied by Trudinger and Wang \cite{15} and Labutin \cite{26}. Now, in a similar way as in \cite{27} and \cite{28}, we define:
$$\mathcal{E}_{m}^0(\Omega)=\left\{u\in\mathscr{C}^{-}_m(\Omega)\cap L^\infty(\Omega);\ {\ds\lim_{x\rightarrow\partial\Omega}}u(x)=0,\ \int_{\Omega\times\rb^n}(dd^{\#}u)^m\wedge\beta^{n-m}<+\infty\right\},$$
$$\mathcal{F}_m(\Omega)=\left\{u\in\mathscr{C}^{-}_m(\Omega);\ \exists(u_j)_j\subset\mathcal{E}_{m}^0(\Omega),\ u_j\downarrow u,\ \sup_j\int_{\Omega\times\rb^n}(dd^{\#}u_j)^m\wedge\beta^{n-m}<+\infty\right\}.$$
Note that $\mathcal{E}_{m}^0(\Omega)$ is a subclass of the one defined and used by Wan \cite{31} for proper $m$-convex functions in order to prove several estimates for the mixed $m$-Hessian operator. According to \cite{19} and \cite{15} it was proved that for every bounded functions $g,h\in\mathscr{C}_m(\Omega)$,  we have:
\begin{equation}\label{eq1}
g=h\ \mathrm{on}\ \partial\Omega\ \mathrm{and}\ g\leqslant h\ \mathrm{in}\ \Omega\Rightarrow\int_{\Omega\times\rb^n}(dd^{\#}h)^m\wedge\beta^{n-m}\leqslant\int_{\Omega\times\rb^n}(dd^{\#}g)^m\wedge\beta^{n-m}.
\end{equation}
For the discontinuous functions, the equality $g=h$ is understood in the sense of limit. The following proposition was given by Cegrell [4] for negative plurisubharmonic functions and extended by Lu \cite{27} for $m$-subharmonic functions.
\begin{pro}\label{pr11}\
\begin{enumerate}
\item $\mathcal{E}_{m}^0(\Omega)\subset\mathcal{F}_m(\Omega)\subset\mathscr{C}^{-}_m(\Omega)$.
\item If $u\in\mathcal{F}_m(\Omega)$, then $\ds\int_{\Omega\times\rb^n}(dd^{\#}u)^m\wedge\beta^{n-m}<+\infty$.
\item The class $\mathcal{E}_{m}^0(\Omega)$ (as well as $\mathcal{F}_m(\Omega)$) is convex cone.
\item If $\varphi\in\mathcal{E}^{0}_m(\Omega)$ (respectively in $\mathcal{F}_m(\Omega)$) and $\psi\in\mathscr{C}^{-}_m(\Omega)$, then $\max(\varphi,\psi)\in\mathcal{E}^{0}_m(\Omega)$ (respectively in $\mathcal{F}_m(\Omega)$).
\end{enumerate}
\end{pro}
\begin{proof} It is clear that (1) and (2) are immediate from the definitions and Theorem 2.4 in \cite{15}. So, all we need is to prove (3) and (4). It is not hard to see that if $\varphi\in\mathcal{E}^{0}_m(\Omega)$ then $\alpha\varphi\in\mathcal{E}^{0}_m(\Omega)$, $\forall\alpha\in\rb^+$. Now, if $\varphi,\psi\in\mathcal{E}^{0}_m(\Omega)$, then $$\int_{\Omega\times\rb^n}(dd^{\#}(\varphi+\psi))^m\wedge\beta^{n-m}=\int_{\{\varphi<\psi\}\times\rb^n}(dd^{\#}(\varphi+\psi))^m\wedge\beta^{n-m}+\int_{\{\psi\leqslant\varphi\}\times\rb^n}(dd^{\#}(\varphi+\psi))^m\wedge\beta^{n-m}.$$
We have $\{\varphi<\psi\}\subset\{2\varphi<\varphi+\psi\}$, then by (\ref{eq1}), we obtain
$$\begin{array}{lcl}
\ds\int_{\{\varphi<\psi\}\times\rb^n}(dd^{\#}(\varphi+\psi))^m\wedge\beta^{n-m}&\leqslant&\ds\int_{\{2\varphi<\varphi+\psi\}\times\rb^n}(dd^{\#}(\varphi+\psi))^m\wedge\beta^{n-m}\\&\leqslant&\ds\int_{\{2\varphi<\varphi+\psi\}\times\rb^n}(dd^{\#}(2\varphi))^m\wedge\beta^{n-m}\\&=&2^m\ds\int_{\{2\varphi<\varphi+\psi\}\times\rb^n}(dd^{\#}\varphi)^m\wedge\beta^{n-m}\\&\leqslant&2^m\ds\int_{\Omega\times\rb^n}(dd^{\#}\varphi)^m\wedge\beta^{n-m}<+\infty.
\end{array}$$
Let $\lambda>1$, then $\lambda\psi<\psi$ and $\{\psi\leqslant\varphi\}\subset\{\lambda\psi<\varphi\}\subset\{(\lambda+1)\psi<\varphi+\psi\}$. Hence, the implication (\ref{eq1}) implies that
$$\int_{\{\psi\leqslant\varphi\}\times\rb^n}(dd^{\#}(\varphi+\psi))^m\wedge\beta^{n-m}\leqslant(\lambda+1)^m\int_{\Omega\times\rb^n}(dd^{\#}\psi)^m\wedge\beta^{n-m}<+\infty.$$
Thus, $$\int_{\Omega\times\rb^n}(dd^{\#}(\varphi+\psi))^m\wedge\beta^{n-m}<+\infty,$$ and then $\varphi+\psi\in\mathcal{E}^{0}_m(\Omega)$. Finally, for $\varphi\in\mathcal{E}^{0}_m(\Omega)$, $\psi\in\mathscr{C}^{-}_m(\Omega)$ and $\lambda>1$, the implication (\ref{eq1}) yields
$$\begin{array}{lcl}
\ds\int_{\Omega\times\rb^n}(dd^{\#}\max(\varphi,\psi))^m\wedge\beta^{n-m}&=&\ds\int_{\{\lambda\varphi<\max(\varphi,\psi)\}\times\rb^n}(dd^{\#}\max(\varphi,\psi))^m\wedge\beta^{n-m}\\&\leqslant&\lambda^m\ds\int_{\{\lambda\varphi<\max(\varphi,\psi)\}\times\rb^n}(dd^{\#}\varphi)^m\wedge\beta^{n-m}\\&=&\lambda^m\ds\int_{\Omega\times\rb^n}(dd^{\#}\varphi)^m\wedge\beta^{n-m}<+\infty,
\end{array}$$
and then $\max(\varphi,\psi)\in\mathcal{E}^{0}_m(\Omega)$. For the other class $\mathcal{F}_m(\Omega)$, the properties (3) and (4) can be proved by repeating almost the same arguments as above.
\end{proof}
\begin{lem} Let $u\in\mathscr{C}^{-}_m(\Omega)\cap\mathcal{C}(\Omega)$, $K\Subset\Omega$ and 
$$u_K=\sup\{v\in\mathscr{C}^{-}_m(\Omega);\ v\leqslant u\ {\rm on}\ K\}.$$
Then, $\Supp (dd^{\#}\overline{u}_K)^m\wedge\beta^{n-m}\subset\overline{K}$, where $\overline{u}_K$ is the upper semi-continuous regularization of $u_K$.
\end{lem}
\begin{proof} By Dini's theorem,
$$u_K=\sup\{v\in\mathscr{C}^{-}_m(\Omega)\cap\mathcal{C}(\Omega);\ v\leqslant u\ {\rm on}\ K\}.$$ 
Hence, Choquet's Lemma implies that there is an increasing sequence $(u_j)_j\subset\mathscr{C}^{-}_m(\Omega)\cap\mathcal{C}(\Omega)$ with $u_j\leqslant u,\ \forall j$ on $K$ and $\overline{u}_K=\overline{\sup_j(u_j)}$. If $B$ is a ball in $\Omega\smallsetminus K$, then by Theorem 1.1 in \cite{01} (The Dirichlet problem) we can take $w_j$ to be the unique continuous $m$-convex negative function on $B$ with $w_j=u_j$ on $\partial B\subset\Omega$ and $(dd^{\#}w_j)^m\wedge\beta^{n-m}=0$ in $B$. Next, we set 
$$\tilde{u}_j=\left\lbrace\begin{array}{lll}
u_j\quad&\mathrm{on}\quad\Omega\smallsetminus B&\\
w_j\quad&\mathrm{on}\quad B.&
\end{array}\right.$$
Then, by Proposition 2.2 in \cite{17}, $\tilde{u}_j\in\mathscr{C}^{-}_m(\Omega)$. Also, by Theorem 3.1 in \cite{19}, $\tilde{u}_j\geqslant u_j$ and the sequence $(\tilde{u}_j)_j$ is increasing. Hence, $\overline{u}_K=\overline{\sup_j(\tilde{u}_j)}$ and due to Theorem 2.4 in \cite{15}, the sequence $(dd^{\#}\tilde{u}_j)^m\wedge\beta^{n-m}$ converges to $(dd^{\#}\overline{u}_K)^m\wedge\beta^{n-m}$. Thus, since $(dd^{\#}\tilde{u}_j)^m\wedge\beta^{n-m}=0$ on $B$, then $(dd^{\#}\overline{u}_K)^m\wedge\beta^{n-m}=0$ on $B$. Furthermore, since $B$ is arbitrary on $\Omega\smallsetminus K$, then $(dd^{\#}\overline{u}_K)^m\wedge\beta^{n-m}=0$ on $\Omega\smallsetminus K$.
\end{proof}
\begin{pro}\label{PP} Let $u\in\mathscr{C}^{-}_m(\Omega)$, then $u$ is locally in $\mathcal{F}_m(\Omega)$; i.e.  for all $K\Subset\Omega$, there is a function $u_K\in\mathcal{F}_m(\Omega)$ such that $u=u_K$ on $K$. 
\end{pro}
\begin{proof} Let $u\in\mathscr{C}^{-}_m(\Omega)$. Thanks to Theorem 1.1 in \cite{31}, we can find $(u_j)_j\subset\mathcal{E}^{0}_m(\Omega)\cap\mathcal{C}(\Omega)$ such that $u_j$ decreases to $u$ on $\Omega$. Let $K\Subset\Omega$ and $$u_{j,K}=\sup\{v\in\mathscr{C}^{-}_m(\Omega);\ v\leqslant u_j\ {\rm on}\ K\},\ \forall j.$$
Then, $\overline{u}_{j,K}\in\mathcal{E}^{0}_m(\Omega)$ and $\Supp (dd^{\#}\overline{u}_{j,K})^m\wedge\beta^{n-m}\subset\overline{K},\ \forall j$. Moreover, $\overline{u}_{j,K}$ decreases to $u$ on $K$, $\overline{u}_{j,K}$ decreases on $\Omega$ and $\overline{u}_{j,K}\geqslant u$ everywhere on $\Omega$. Thus, $\lim_{j\rightarrow+\infty}\overline{u}_{j,K}=u_{K}\geqslant u$. Therefore, since $(dd^{\#}\overline{u}_{j,K})^m\wedge\beta^{n-m}$ is weakly convergent, it follows that $$\sup_{j}\int_{\Omega\times\rb^n}(dd^{\#}\overline{u}_{j,K})^m\wedge\beta^{n-m}<+\infty.$$
Hence, $u_{K}\in\mathcal{F}_m(\Omega)$ such that $u=u_{K}$ on $K$.
\end{proof}
\begin{defn} A set $E\subset\rb^n$ is said to be $m$-polar, if for each point $a\in E$ there exists a neighborhood $V_a$ of $a$ and a function $u\in\mathscr{C}_m(V_a)$ such that $u=-\infty$ on $E\cap V_a$.
\end{defn}
According to Trindinger and Wang \cite{15}, we can define the $m$-Hessian capacity of a compact subset $K\subset\Omega$ by 
$$cap_m(K,\Omega)=\sup\left\lbrace\int_{K\times\rb^n}(dd^{\#}u)^m\wedge\beta^{n-m};\ u\in\mathscr{C}_m(\Omega),\ -1\leqslant u\leqslant 0\right\rbrace.$$ 
Furthermore, for an arbitrary subset $E\subset\Omega$, we define 
$$cap_m(E,\Omega)=\inf\left\lbrace cap_m(\omega,\Omega);\ \omega\ {\rm is\ open},\ E\subset\omega\subset\Omega\right\rbrace,$$ 
where, $cap_m(\omega,\Omega)=\sup\left\lbrace cap_m(K,\Omega) ;\ K\ {\rm compact},\ K\subset\omega\right\rbrace$.
It was proved by Labutin \cite{26} that every Borel subset $E\subset\Omega$ is {\it capacitable}, that is 
$$cap_m(E,\Omega)=\sup\left\lbrace cap_m(K,\Omega) ;\ K\ {\rm compact},\ K\subset E\right\rbrace.$$
The second main tool in potential theory is the relative extremal function defined as 
$$R_m(E,\Omega)(x):=\sup\lbrace u(x);\ u\in\mathscr{C}^{-}_m(\Omega),\ u\leqslant-1\ {\rm on}\ E\rbrace,\quad x\in\Omega\ {\rm and}\ E\Subset\Omega.$$
By using the same notations as in \cite{26}, we denote by $\overline R_m(E,\Omega)$ the upper semi-continuous regularisation of $R_m(E,\Omega)$, which is negative and $m$-convex in $\Omega$. Moreover, we have $\overline R_m(E,\Omega)=R_m(E,\Omega)$ almost everywhere in $\Omega$. The following proposition is due to Labutin \cite{26} and can be seen as the corresponding result of the one obtained by Lu \cite{27} in the complex Hessian setting:
\begin{pro}\label{pop}\
\begin{enumerate} 
\item Let $R>0$ and $\Omega=\mathbb{B}(0,R)$. If $E\Subset\mathbb{B}(R)$, then $E$ is $m$-polar if and only if $cap_m(E,B_R)=0$ if and only if $\overline R_m(E,\mathbb{B}(R))=0$.
\item Let $E\Subset\rb^n$ be $m$-polar. Then, there exists $u\in\mathscr{C}_m(\rb^n)$ such that $u=-\infty$ on $E$.
\end{enumerate}
\end{pro}
 Now, strongly inspired by Nguyen \cite{25}, we introduce the following definition of weighted relative extremal function:
\begin{defn} Let $E\subset\Omega$ and $u\in\mathscr{C}^{-}_m(\Omega)$. The weighted relative extremal function associated to $E$ and $u$ is defined by $$R_{m,u}(E)=R_{m,u}(E,\Omega):=\sup\{v\in\mathscr{C}^{-}_m(\Omega);\ v\leqslant u\ \mathrm{outside\ an\ \mathit{m-}polar\ subset\ on}\ E\}.$$
\end{defn}
Denote by $\overline{R}_{m,u}(E)$ the upper semi-continuous regularization of $R_{m,u}(E)$. By using the second statement of Proposition \ref{pop}, it is not hard to see that $\overline R_{m,u}(E)=\overline R_{m,u}(E\smallsetminus F)$, for all $m$-polar $F\subset\Omega$. This means in particular that when $u$ is identically $-1$, we recover the relative extremal function as defined by Labutin \cite{26}. A function $u\in\mathscr{C}_m(\Omega)$ is called $m$-maximal if for each $v\in\mathscr{C}_m(\Omega)$ such that $v\leqslant u$ outside a compact subset of $\Omega$ implies that $v\leqslant u$ in $\Omega$. In particularly, by Theorem A.1 in \cite{17}, if $u\in\mathcal{C}(\Omega)$ then the function $u$ is $m$-maximal if and only if $(dd^{\#}u)^m\wedge\beta^{n-m}=0$. Thanks to Theorem 1.1 in \cite{31} and the proof of Theorem A.1 in \cite{17} it is not hardto seethat each $m$-maximal bounded convex function can be approximated locally by a decreasing sequence of continuous $m$-maximal function. Hence, in view of Theorem 2.4 in \cite{15}, we obtain the following corresponding of a well-known result in the complex hessian setting. 
\begin{thm}[maximality criterion]\label{mc} A function $u\in\mathscr{C}_m(\Omega)\cap L_{loc}^\infty(\Omega)$ is $m$-maximal if and only if $(dd^{\#}u)^m\wedge\beta^{n-m}=0$.
\end{thm}
\begin{pro}\label{pr3} Let $u\in\mathscr{C}^{-}_m(\Omega)$ and $E\subset\Omega$, then:
\begin{enumerate}
\item $R_{m,u}(E)\in\mathscr{C}^{-}_m(\Omega)$.
\item $R_{m,u}(E)$ is $m$-maximal on $\Omega\smallsetminus\overline{E}$.
\item $E\Subset\Omega\Rightarrow R_{m,u}(E)\in\mathcal{F}_m(\Omega).$
\item Let $(u_j)_j\subset\mathscr{C}^{-}_m(\Omega)$ be a decreasing sequence converging to $u$. Then, the sequence $(R_{m,u_j}(E))_j$ decreases to $R_{m,u}(E)$.
\end{enumerate}
\end{pro}
\begin{proof} (1) It is not hard to see that $R_{m,u}(E)=\overline{R}_{m,u}(E)$ outside an $m$-polar subset of $\Omega$. Thus, by definition of $R_{m,u}(E)$, we can see that $R_{m,u}(E)=\overline{R}_{m,u}(E)\in\mathscr{C}^{-}_m(\Omega)$.

(2) Let $v\in\mathscr{C}^{-}_m(\Omega\smallsetminus\overline{E})$ and $v\leqslant R_{m,u}(E)$ outside a compact
subset of $\Omega\smallsetminus\overline{E}$. Let us set
$$\varphi=\left\lbrace\begin{array}{lll}
\sup\{v,R_{m,u}(E)\}\quad&\mathrm{on}\quad\Omega\smallsetminus\overline{E}&\\
R_{m,u}(E)\quad&\mathrm{on}\quad \overline{E}.&
\end{array}\right.$$
Thus, by Proposition 2.2 in \cite{17}, $\varphi\in\mathscr{C}^{-}_m(\Omega)$ and $\varphi\leqslant R_{m,u}(E)$ on $\Omega$. Hence, $\varphi=R_{m,u}(E)$ on $\Omega$, which implies that $v\leqslant R_{m,u}(E)$ on $\Omega\smallsetminus\overline{E}$.

(3) Due to Proposition \ref{PP}, there exists a decreasing sequence $(u_j)_j\subset\mathcal{E}_{m}^0(\Omega)$ converging to $u$ on a neighborhood of $\overline{E}$ such that $$\sup_j\int_{\Omega\times\rb^n}(dd^{\#}u_j)^m\wedge\beta^{n-m}<+\infty.$$ Then, $(R_{m,u_j}(E))_j\subset\mathcal{E}_{m}^0(\Omega)$ is a decreasing sequence such that for all $j$, $u_j\leqslant R_{m,u_j}(E)$ on $\Omega$ and $u_j=R_{m,u_j}(E)$ on $E$ outside an $m$-polar set. Thus, $(R_{m,u_j}(E))_j$ converges to a function $\varphi\in\mathscr{C}^{-}_m(\Omega)$ such that $R_{m,u}(E)\leqslant\varphi$ on $\Omega$ and $\varphi=R_{m,u}(E)$ on $E$ outside an $m$-polar set. Hence, $\varphi=R_{m,u}(E)$ on $\Omega$. Moreover, by (\ref{eq1}), we have $$\int_{\Omega\times\rb^n}(dd^{\#}R_{m,u_j}(E))^m\wedge\beta^{n-m}\leqslant\int_{\Omega\times\rb^n}(dd^{\#}u_j)^m\wedge\beta^{n-m}.$$ Thus, the sequence $(R_{m,u_j}(E))_j$ decreases to $R_{m,u}(E)$ and $$\sup_j\int_{\Omega\times\rb^n}(dd^{\#}R_{m,u_j}(E))^m\wedge\beta^{n-m}\leqslant\sup_j\int_{\Omega\times\rb^n}(dd^{\#}u_j)^m\wedge\beta^{n-m}<+\infty,$$ which implies that $R_{m,u}(E)\in\mathcal{F}_m(\Omega)$.

(4) Assume that $(R_{m,u_j}(E))_j$ decreases to $\varphi\in\mathscr{C}^{-}_m(\Omega)$. We have $R_{m,u}(E)\leqslant R_{m,u_j}(E)$ on $\Omega$, then  $R_{m,u}(E)\leqslant\varphi$. Moreover, $u_j=R_{m,u_j}(E)$ and $u=R_{m,u}(E)$ on $E$ outside an $m$-polar set, thus $\varphi=R_{m,u}(E)$.
\end{proof}
\begin{defn} For each Borel subset $E\subset \Omega$ and $u\in\mathscr{C}^{-}_m(\Omega)$, we define the weighted $m$-Hessian capacity associated to $u$ of $E$ by $$cap_{m,u}(E)=\sup\left\lbrace\int_{K\times\rb^n}(dd^{\#}v)^m\wedge\beta^{n-m};\ K\Subset E,\ v\in\mathscr{C}_m^-(\Omega)\cap L^\infty(\Omega)\ \mathrm{and}\ u\leqslant v\leqslant0\right\rbrace.$$
\end{defn}
In particular, if $u\equiv -1$ and $E\Subset \Omega$, we recover the $m$-Hessian capacity investigated by Trindinger and Wang \cite{15}. Moreover, thanks to Proposition \ref{cln}, if $E\Subset \Omega$ then the capacity $cap_{m,u}(E)$ is finite provided that $u$ is locally bounded.
\begin{pro}\label{pr4} Assume that $u_j,u\in\mathscr{C}_m^-(\Omega)$ such that $(u_j)_j$ is monotone decreasing to $u$. Then, the sequence $(cap_{m,u_j}(E))_j$ is monotone increasing to $cap_{m,u}(E)$, for all Borel subset $E$ of $\Omega$ such that $cap_{m,u}(E)<+\infty$.
\end{pro}
\begin{proof} It is not hard to see that $(cap_{m,u_j}(E))_j$ is an increasing sequence and $cap_{m,u_j}(E)$ increases to $c\leqslant cap_{m,u}(E)$ as $j\rightarrow+\infty$. Let $v\in\mathscr{C}_m^-(\Omega)\cap L^\infty(\Omega)$ such that $u\leqslant v\leqslant0$. By Theorem 4.3 in \cite{26} and by using the subaddivity of $cap_m$, it is not difficult to realize that for each $\varepsilon>0$ there exists an open subset $U\subset\Omega$ such that $cap_m(U,\Omega)<\varepsilon$ and $u_j,u$ are continuous on $\Omega\smallsetminus U$, for each $j$. Let us take an open subset $\Omega_0$ such that $E\subset\Omega_0\Subset\Omega$. So, according to Dini's theorem, $(u_j)_j$ uniformly converges to $u$ on $\Omega_0\smallsetminus U$. Thus, we can choose $j_0$ such that $u_{j_0}<(1-\varepsilon)u\leqslant(1-\varepsilon)v$ on $\Omega_0\smallsetminus U$. Then,
$$\begin{array}{lcl}
\ds c\geqslant cap_{m,u_{j_0}}(E)&\geqslant&\ds\int_{E\times\rb^n}(dd^{\#}\max((1-\varepsilon)v,u_{j_0}))^m\wedge\beta^{n-m}\\&\geqslant&\ds\int_{E\smallsetminus U\times\rb^n}(dd^{\#}\max((1-\varepsilon)v,u_{j_0}))^m\wedge\beta^{n-m}\\&=&\ds(1-\varepsilon)^m\int_{E\smallsetminus U\times\rb^n}(dd^{\#}v)^m\wedge\beta^{n-m}\\&\geqslant&\ds(1-\varepsilon)^m\left(\int_{E\times\rb^n}(dd^{\#}v)^m\wedge\beta^{n-m}-\int_{U\times\rb^n}(dd^{\#}v)^m\wedge\beta^{n-m}\right)\\&\geqslant&\ds(1-\varepsilon)^m\left(\int_{E\times\rb^n}(dd^{\#}v)^m\wedge\beta^{n-m}-\left(\sup_U|v|\right)^m cap_m(U,\Omega)\right)\\&\geqslant&\ds(1-\varepsilon)^m\int_{E\times\rb^n}(dd^{\#}v)^m\wedge\beta^{n-m}-(1-\varepsilon)^m\varepsilon\left(\sup_U|v|\right)^m.
\end{array}$$
Hence, by letting $\varepsilon\rightarrow0$, we obtain $$c\geqslant\int_{E\times\rb^n}(dd^{\#}v)^m\wedge\beta^{n-m}.$$ Thus, $c\geqslant cap_{m,u}(E)$, which leads to the conclusion that $c=cap_{m,u}(E)$.
\end{proof}
\begin{lem}\label{ll} Let $u,v\in\mathscr{C}_m^-(\Omega)\cap L_{loc}^\infty(\Omega)$, then
$$(dd^{\#}\max(u,v))^m\wedge\beta^{n-m}\geqslant{\1}_{\lbrace u\geqslant v\rbrace}(dd^{\#}u)^m\wedge\beta^{n-m}+{\1}_{\lbrace v>u\rbrace}(dd^{\#}v)^m\wedge\beta^{n-m}.$$
\end{lem}
\begin{proof} For all compact $K\subset\lbrace u\geqslant v\rbrace$, we have
$$\begin{array}{lcl}
\ds\int_{K\times\rb^n}(dd^{\#}\max(u,v))^m\wedge\beta^{n-m}&\geqslant&\ds\limsup_{\varepsilon\rightarrow0}\ds\int_{K\times\rb^n}(dd^{\#}\max(u+\varepsilon,v))^m\wedge\beta^{n-m}\\&=&\ds\limsup_{\varepsilon\rightarrow0}\ds\int_{K\times\rb^n}{\1}_{\lbrace u+\varepsilon>v\rbrace}(dd^{\#}\max(u+\varepsilon,v))^m\wedge\beta^{n-m}\\&=&\ds\limsup_{\varepsilon\rightarrow0}\ds\int_{K\times\rb^n}{\1}_{\lbrace u+\varepsilon>v\rbrace}(dd^{\#}u)^m\wedge\beta^{n-m}\\&=&\ds\int_{K\times\rb^n}(dd^{\#}u)^m\wedge\beta^{n-m}.
\end{array}$$
The other term is then obtained by
reversing the roles of $u$ and $v$.
\end{proof}
\begin{thm}[Integration by parts]\label{th1} Let the functions $v,u_1,...,u_m\in\mathscr{C}_m^-(\Omega)$ and denote by $T=dd^{\#}u_2\wedge...\wedge dd^{\#}u_m\wedge\beta^{n-m}$, then:
\begin{enumerate}
\item If $v,u_1\in L^{\infty}_{loc}(\Omega)$ and $u=v$ outside a compact subset of $\Omega$, then 
$$\int_{\Omega\times\rb^n}vdd^{\#}u_1\wedge T=\int_{\Omega\times\rb^n}u_1dd^{\#}v\wedge T.$$
\item If  $v,u_1\in L^{\infty}_{loc}(\Omega)$, $v$ is an exhaustion function for $\Omega$ and $\ds\int_{\Omega\times\rb^n}dd^{\#}v\wedge T<+\infty$, then $$\int_{\Omega\times\rb^n}vdd^{\#}u_1\wedge T\geqslant\int_{\Omega\times\rb^n}u_1dd^{\#}v\wedge T.$$
\item If $v\not\equiv0$, $\lim_{x\rightarrow y}v(x)=0,\ \forall y\in\partial\Omega$ and $\ds\int_{\Omega\times\rb^n}u_1dd^{\#}v\wedge T>-\infty$, then $$\int_{\Omega\times\rb^n}vdd^{\#}u_1\wedge T\geqslant\int_{\Omega\times\rb^n}u_1dd^{\#}v\wedge T,$$ and the equality will holds if we also assume that $\lim_{x\rightarrow y}u_1(x)=0,\ \forall y\in\partial\Omega$.
\end{enumerate}
\end{thm}
Theorem \ref{th1} is the corresponding of Theorems 3.1 and 3.3 in \cite{02} and Theorem 3.2 in \cite{28}.
\begin{proof}
(1) We assume first that $v,u_1,...,u_m$ are defined in a neighborhood of $\partial\Omega$, $v$ and $u_1$ are smooth and $v=u_1$ in a neighborhood
of $\partial\Omega$. Let $K$ be a compact subset of $\Omega$ such that $\{x\in\Omega;\ u_1(x)\neq v(x)\}\Subset\mathring{K}$. Let $\chi$ be a smooth and compactly supported function such that $0\leqslant\chi\leqslant1$ on $\Omega$ and $\chi\equiv1$ on $K$. Then
$$\begin{array}{lcl}
\ds\int_{\Omega\times\rb^n}\chi vdd^{\#}u_1\wedge T=\langle dd^{\#}u_1T,\chi v\rangle&=&\langle u_1T,dd^{\#}(\chi v)\rangle\\&=&\langle u_1T,dd^{\#}(\chi v)-\chi dd^{\#}v\rangle+\langle u_1T,\chi dd^{\#}v\rangle.
\end{array}$$
Note that $dd^{\#}(\chi v)-\chi dd^{\#}v=vdd^{\#}\chi+d\chi\wedge d^{\#}v-d^{\#}\chi\wedge dv$ is a smooth form whose support is contained in an open set of $\{u_1=v\}$, thus
$$\begin{array}{lcl}
\langle u_1T,dd^{\#}(\chi v)-\chi dd^{\#}v\rangle&=&\langle vT,dd^{\#}(\chi v)\rangle-\langle vT,\chi dd^{\#}v\rangle\\&=&\langle dd^{\#}v\wedge T,\chi v\rangle-\langle vT,\chi dd^{\#}v\rangle=0.
\end{array}$$
Since $v$ is smooth, we obtain
$$\int_{\Omega\times\rb^n}\chi vdd^{\#}u_1\wedge T=\int_{\Omega\times\rb^n}\chi u_1dd^{\#}v\wedge T,$$ and our result will follows by extending the support of $\chi$ to be $\Omega$. Let us consider now the general case. Let $K$ be a compact subset of $\Omega$ such that $u_1 = v$ in $\Omega\smallsetminus K$ and choose $\lambda>0$ such that $K_{2\lambda}=\{x;\ {\rm dist}(x,K)\leqslant2\lambda\}$ is contained in $\Omega$. For $\varepsilon>0$, let $u_{1,\varepsilon}$ and $v_\varepsilon$ are smooth regularizations of $u$ and $v$ respectively. Then, $u_{1,\varepsilon}=v_\varepsilon$ outside $K_\lambda$ if
$\varepsilon<\lambda$. Finally, we choose $\delta_0>0$ small enough such that that $K_{2\lambda}\subset\Omega_{\delta_0}=\{x\in\Omega;\ {\rm dist}(x,\partial\Omega)>\delta_0\}$. For $\delta<\delta_0$, we have $\max\{v(x);\ x\in\overline{\Omega}_\delta\}<0$ and since $v_\varepsilon$ is smooth and
decreases to $v$ as $\varepsilon\rightarrow0$ it follows that $v_\varepsilon<0$ on $\overline{\Omega}_\delta$, if $\varepsilon$ is small enough. Let $\chi$ be a smooth and compactly supported function such that $0\leqslant\chi\leqslant1$ on $\Omega$ and $\chi\equiv1$ on $K_{2\lambda}$. Then, by the first step, we have 
$$\int_{\Omega_\delta\times\rb^n}\chi v_\varepsilon dd^{\#}u_{1,\varepsilon}\wedge T=\int_{\Omega_\delta\times\rb^n}\chi u_{1,\varepsilon}dd^{\#}v_\varepsilon\wedge T,$$
which implies, by letting $\varepsilon\rightarrow0$ and using the weak convergence of $v_\varepsilon dd^{\#}u_{1,\varepsilon}\wedge T$ and $u_{1,\varepsilon}dd^{\#}v_\varepsilon\wedge T$ (see Theorem 2.6 in \cite{15}), that
$$\int_{\Omega_\delta\times\rb^n}\chi vdd^{\#}u_1\wedge T=\int_{\Omega_\delta\times\rb^n}\chi u_1dd^{\#}v\wedge T.$$
Hence, 
$$\int_{\Omega_\delta\times\rb^n}vdd^{\#}u_1\wedge T=\int_{\Omega_\delta\times\rb^n}u_1dd^{\#}v\wedge T,$$
and we conclude by sending $\delta$ to $0$.

(2) Let $\varepsilon>0$, $k\in\nb^{\ast}$ and $u_{1,k}=\max(u_1-\varepsilon,kv)$. Then, $u_{1,k}\in\mathscr{C}_m^-(\Omega)$ and since $v$ is an exhaustion function for $\Omega$, we have $u_{1,k}=kv$ near $\partial\Omega$. Thus, by (1), we get 
$$\int_{\Omega\times\rb^n}vdd^{\#}u_{1,k}\wedge T=\int_{\Omega\times\rb^n}u_{1,k}dd^{\#}v\wedge T.$$
In $\Omega$, $u_{1,k}$ decreases to $u_1-\varepsilon$ as $k\rightarrow+\infty$, and by  the monotone convergence theorem, we have 
$$\int_{\Omega\times\rb^n}(u_1-\varepsilon)dd^{\#}v\wedge T=\lim_{k\rightarrow+\infty}\int_{\Omega\times\rb^n}vdd^{\#}u_{1,k}\wedge T.$$
Since $vdd^{\#}u_{1,k}\wedge T$ converges to $vdd^{\#}u_{1}\wedge T$ as $k\rightarrow+\infty$, it follows that
$$\lim_{k\rightarrow+\infty}\int_{\Omega\times\rb^n}\chi vdd^{\#}u_{1,k}\wedge T=\int_{\Omega\times\rb^n}\chi vdd^{\#}u_{1}\wedge T,$$
for every smooth and compactly supported function $\chi$ on $\Omega$ such that $0\leqslant\chi\leqslant1$ (see Theorem 2.6 in \cite{15}). Note also that
$$\int_{\Omega\times\rb^n}\chi vdd^{\#}u_{1,k}\wedge T\geqslant\int_{\Omega\times\rb^n} vdd^{\#}u_{1,k}\wedge T.$$
Hence, we conclude that 
$$\int_{\Omega\times\rb^n}(u_1-\varepsilon)dd^{\#}v\wedge T\leqslant\int_{\Omega\times\rb^n}\chi vdd^{\#}u_{1}\wedge T.$$
Thus, 
$$\int_{\Omega\times\rb^n}(u_1-\varepsilon)dd^{\#}v\wedge T\leqslant\int_{\Omega\times\rb^n}vdd^{\#}u_{1}\wedge T,$$
and since the positive measure $dd^{\#}v\wedge T$ is finite on $\Omega$, then the desired result will follows by letting $\varepsilon\rightarrow0$.

(3) Assume that $v,u_1\in\mathcal{C}(\overline{\Omega})$ and $v=u_1=0$ on $\partial\Omega$. Then, by (2), we have $$-\infty<\int_{\Omega\times\rb^n}u_1dd^{\#}v\wedge T\leqslant\int_{\{v<-\varepsilon\}\times\rb^n}u_1dd^{\#}v\wedge T\leqslant\int_{\{v<-\varepsilon\}\times\rb^n}(v+\varepsilon)dd^{\#}u_1\wedge T.$$
Thus, if we denote by ${\1}_\varepsilon$ the characteristic function of $\{v<-\varepsilon\}$, then $(v+\varepsilon){\1}_\varepsilon$ decreases to $v$ when $\varepsilon\rightarrow0$. Hence, 
$$\int_{\Omega\times\rb^n}u_1dd^{\#}v\wedge T\leqslant\int_{\Omega\times\rb^n}vdd^{\#}u_1\wedge T.$$
In the general case, we use Theorem 1.1 in \cite{31} and choose $v_j,u_{1,j}\in\mathcal{E}^{0}_m(\Omega)\cap\mathcal{C}(\overline{\Omega})$ such that $v_j$ and $u_{1,j}$ decrease respectively to $v$ and $u_1$ as $j\rightarrow+\infty$. Then, by Theorem 2.6 in \cite{15}, we have $$\begin{array}{lcl}
\ds\int_{\Omega\times\rb^n}u_{1}dd^{\#}v\wedge T&\leqslant&\ds\int_{\Omega\times\rb^n}u_{1,k}dd^{\#}v\wedge T\\&\leqslant&\ds\lim_{j\rightarrow+\infty}\int_{\Omega\times\rb^n}u_{1,k}dd^{\#}v_j\wedge T\\&=&\ds\lim_{j\rightarrow+\infty}\int_{\Omega\times\rb^n}v_jdd^{\#}u_{1,k}\wedge T=\int_{\Omega\times\rb^n}vdd^{\#}u_{1,k}\wedge T.
\end{array}$$ 
Since $v$ is negative and upper semi-continuous and $dd^{\#}u_{1,k}\wedge T$ is weakly convergent to $dd^{\#}u_{1}\wedge T$, then for all $\varepsilon>0$, we get 
$$\begin{array}{lcl}
\ds\int_{\Omega\times\rb^n}u_{1}dd^{\#}v\wedge T&\leqslant&\ds\lim_{k\rightarrow+\infty}\int_{\Omega\times\rb^n}vdd^{\#}u_{1,k}\wedge T\\&\leqslant&\ds\lim_{k\rightarrow+\infty}\int_{\{v<-\varepsilon\}\times\rb^n}vdd^{\#}u_{1,k}\wedge T\\&\leqslant&\ds\int_{\{v<-\varepsilon\}\times\rb^n}vdd^{\#}u_{1}\wedge T.
\end{array}$$ 
Hence, by letting $\varepsilon\rightarrow0$, we have 
$$\int_{\Omega\times\rb^n}u_{1}dd^{\#}v\wedge T\leqslant\int_{\Omega\times\rb^n}vdd^{\#}u_1\wedge T,$$
which concludes our proof. 
\end{proof}
\begin{thm}\label{l2} Let $u_1,...,u_m\in\mathcal{F}_m(\Omega)$ and $u^{j}_1,...,u^{j}_m$ are sequences of functions in $\mathcal{E}^{0}_m(\Omega)$ which decrease respectively to $u_1,...,u_m$. Then, for all $v\in\mathscr{C}_m^-(\Omega)$, we have
$$\lim_{j\rightarrow+\infty}\int_{\Omega\times\rb^n}v dd^{\#}u^{j}_1\wedge...\wedge dd^{\#}u^{j}_m\wedge\beta^{n-m}=\int_{\Omega\times\rb^n}v dd^{\#}u_1\wedge...\wedge dd^{\#}u_m\wedge\beta^{n-m}.$$
\end{thm}
Theorem \ref{l2} is the corresponding of Proposition 5.1 in \cite{28}.
\begin{proof} Assume first that 
\begin{equation}\label{E}
\sup_{k,j}\int_{\Omega\times\rb^n}(dd^{\#}u_{k}^{j})^m\wedge\beta^{n-m}<+\infty.
\end{equation} Then, according to Corollary 1.1 in \cite{31}, it is clear that
$$\sup_{j}\int_{\Omega\times\rb^n}dd^{\#}u^{j}_1\wedge...\wedge dd^{\#}u^{j}_m\wedge\beta^{n-m}<+\infty.$$
Thanks to Theorem 2.4 in \cite{15}, the sequence $dd^{\#}u^{j}_1\wedge...\wedge dd^{\#}u^{j}_m\wedge\beta^{n-m}$ is weakly convergent to $dd^{\#}u_1\wedge...\wedge dd^{\#}u_m\wedge\beta^{n-m}$. Hence, we have
$$+\infty>\lim_{j\rightarrow+\infty}\int_{\Omega\times\rb^n}dd^{\#}u^{j}_1\wedge...\wedge dd^{\#}u^{j}_m\wedge\beta^{n-m}\geqslant\int_{\Omega\times\rb^n}dd^{\#}u_1\wedge...\wedge dd^{\#}u_m\wedge\beta^{n-m}.$$
If $v\in\mathcal{E}^{0}_m(\Omega)\cap\mathcal{C}(\overline\Omega)$, then by Theorem \ref{th1}, we have that $$\ds\int_{\Omega\times\rb^n}vdd^{\#}u_{1}^{j}\wedge...\wedge dd^{\#}u_{m}^{j}\wedge\beta^{n-m}\ {\rm decreases.}$$ Moreover, we have
$$\int_{\Omega\times\rb^n}vdd^{\#}u_{1}^{j}\wedge...\wedge dd^{\#}u_{m}^{j}\wedge\beta^{n-m}\geqslant(\inf_{\Omega}v)\sup_j\int_{\Omega\times\rb^n}dd^{\#}u_{1}^{j}\wedge...\wedge dd^{\#}u_{m}^{j}\wedge\beta^{n-m}>-\infty.$$
Consequently, $$\lim_{j\rightarrow+\infty}\int_{\Omega\times\rb^n}vdd^{\#}u_{1}^{j}\wedge...\wedge dd^{\#}u_{m}^{j}\wedge\beta^{n-m}\ {\rm exists.}$$  Now, if $(w_{1}^{j})_j,...,(w_{m}^{j})_j$ are another sequences which decrease respectively to $u_1,...,u_m$, then by using again Theorem \ref{th1}, we get
$$\begin{array}{lcl}
& &\ds\int_{\Omega\times\rb^n}vdd^{\#}w_{1}^{j}\wedge dd^{\#}w_{2}^{j}\wedge...\wedge dd^{\#}w_{m}^{j}\wedge\beta^{n-m}\\&=&\ds\int_{\Omega\times\rb^n}w_{1}^{j}dd^{\#}v\wedge dd^{\#}w_{2}^{j}\wedge...\wedge dd^{\#}w_{m}^{j}\wedge\beta^{n-m}\\&\geqslant&\ds\int_{\Omega\times\rb^n}u_1dd^{\#}v\wedge dd^{\#}w_{2}^{j}\wedge...\wedge dd^{\#}w_{m}^{j}\wedge\beta^{n-m}\\&=&\ds\lim_{j_1\rightarrow+\infty}\int_{\Omega\times\rb^n}u_{1}^{j_1}dd^{\#}v\wedge dd^{\#}w_{2}^{j}\wedge...\wedge dd^{\#}w_{m}^{j}\wedge\beta^{n-m}\\&=&\ds\lim_{j_1\rightarrow+\infty}\int_{\Omega\times\rb^n}w_{2}^{j}dd^{\#}v\wedge dd^{\#}u_{1}^{j_1}\wedge dd^{\#}w_{3}^{j}\wedge...\wedge dd^{\#}w_{m}^{j}\wedge\beta^{n-m}\\&\geqslant&...\geqslant\ds\lim_{j_1\rightarrow+\infty}...\lim_{j_m\rightarrow+\infty}\int_{\Omega\times\rb^n}vdd^{\#}u_{1}^{j_1}\wedge...\wedge dd^{\#}u_{m}^{j_m}\wedge\beta^{n-m}\\&\geqslant&\ds\lim_{j\rightarrow+\infty}\int_{\Omega\times\rb^n}vdd^{\#}u_{1}^{j}\wedge...\wedge dd^{\#}u_{m}^{j}\wedge\beta^{n-m}.
\end{array}$$
Thus, $\lim_{j\rightarrow+\infty}\int_{\Omega\times\rb^n}vdd^{\#}w_{1}^{j}\wedge...\wedge dd^{\#}w_{m}^{j}\wedge\beta^{n-m}$ exists and minorized by $$\lim_{j\rightarrow+\infty}\int_{\Omega\times\rb^n}vdd^{\#}u_{1}^{j}\wedge...\wedge dd^{\#}u_{m}^{j}\wedge\beta^{n-m}.$$
But this is a symmetric situation, then we conclude that the limits are equal, and then
$$\lim_{j\rightarrow+\infty}\int_{\Omega\times\rb^n}v dd^{\#}u^{j}_1\wedge...\wedge dd^{\#}u^{j}_m\wedge\beta^{n-m}=\int_{\Omega\times\rb^n}v dd^{\#}u_1\wedge...\wedge dd^{\#}u_m\wedge\beta^{n-m}.$$
Let us now remove the restriction (\ref{E}). For this aim, we consider $(h_{1}^{j})_j,...,(h_{m}^{j})_j$ are sequences in $\mathcal{E}^{0}_m(\Omega)$ which decrease respectively to $u_1,...,u_m$ and such that $$\sup_{k,j}\int_{\Omega\times\rb^n}(dd^{\#}h_{k}^{j})^m\wedge\beta^{n-m}<+\infty.$$ 
By setting $g_k^j=\max(h_k^j,u_k^j)$ and using the implication $(5.1)$, we get $$\sup_{k,j}\int_{\Omega\times\rb^n}(dd^{\#}g_{k}^{j})^m\wedge\beta^{n-m}<+\infty.$$
Therefore, by the above arguments of integration by parts, it is not hard to see that 
$$\begin{array}{lcl}\ds\lim_{j\rightarrow+\infty}\int_{\Omega\times\rb^n}v dd^{\#}u^{j}_1\wedge...\wedge dd^{\#}u^{j}_m\wedge\beta^{n-m}&\leqslant& \ds\lim_{j\rightarrow +\infty}\int_{\Omega\times\rb^n}v dd^{\#}g^{j}_1\wedge...\wedge dd^{\#}g^{j}_m\wedge\beta^{n-m}\\&=&\ds\int_{\Omega\times\rb^n}v dd^{\#}u_1\wedge...\wedge dd^{\#}u_m\wedge\beta^{n-m}\\&\leqslant&\ds\lim_{j\rightarrow+\infty}\int_{\Omega\times\rb^n}v dd^{\#}u^{j}_1\wedge...\wedge dd^{\#}u^{j}_m\wedge\beta^{n-m}.\end{array}$$
The last inequality because the sequence $v dd^{\#}u^{j}_1\wedge...\wedge dd^{\#}u^{j}_m\wedge\beta^{n-m}$ converge weakly to $v dd^{\#}u_1\wedge...\wedge dd^{\#}u_m\wedge\beta^{n-m}$. Thus, $$\lim_{j\rightarrow+\infty}\int_{\Omega\times\rb^n}v dd^{\#}u^{j}_1\wedge...\wedge dd^{\#}u^{j}_m\wedge\beta^{n-m}=\int_{\Omega\times\rb^n}v dd^{\#}u_1\wedge...\wedge dd^{\#}u_m\wedge\beta^{n-m}.$$
Assume now that $v\in\mathscr{C}_m^-(\Omega)$ and $$-\int_{\Omega\times\rb^n}v dd^{\#}u_1\wedge...\wedge dd^{\#}u_m\wedge\beta^{n-m}<+\infty.$$
For each $j$, we choose $v_j\in\mathcal{E}^{0}_m(\Omega)\cap\mathcal{C}(\Omega)$ decreasing to $v$, $q_j$ and $s_j$ such that
$$\begin{array}{lcl}
\ds-\int_{\Omega\times\rb^n}v dd^{\#}u_1\wedge...\wedge dd^{\#}u_m\wedge\beta^{n-m}&\leqslant&\ds\frac{1}{j}-\int_{\Omega\times\rb^n}v_j dd^{\#}u_1\wedge...\wedge dd^{\#}u_m\wedge\beta^{n-m}\\&\leqslant&\ds\frac{2}{j}-\int_{\Omega\times\rb^n}v_j dd^{\#}u^{q_j}_1\wedge...\wedge dd^{\#}u^{q_j}_m\wedge\beta^{n-m}\\&\leqslant&\ds\frac{2}{j}-\int_{\Omega\times\rb^n}v dd^{\#}u^{q_j}_1\wedge...\wedge dd^{\#}u^{q_j}_m\wedge\beta^{n-m}\\&\leqslant&\ds\frac{4}{j}-\int_{\Omega\times\rb^n}v_{s_j} dd^{\#}u^{q_j}_1\wedge...\wedge dd^{\#}u^{q_j}_m\wedge\beta^{n-m}\\&\leqslant&\ds\frac{4}{j}-\int_{\Omega\times\rb^n}v_{s_j} dd^{\#}u_1\wedge...\wedge dd^{\#}u_m\wedge\beta^{n-m}\\&\leqslant&\ds\frac{4}{j}-\int_{\Omega\times\rb^n}v dd^{\#}u_1\wedge...\wedge dd^{\#}u_m\wedge\beta^{n-m}.
\end{array}$$
Letting $j\rightarrow+\infty$, we get
$$\lim_{j\rightarrow+\infty}\int_{\Omega\times\rb^n}v dd^{\#}u^{j}_1\wedge...\wedge dd^{\#}u^{j}_m\wedge\beta^{n-m}=\int_{\Omega\times\rb^n}v dd^{\#}u_1\wedge...\wedge dd^{\#}u_m\wedge\beta^{n-m}.$$
Also, if
$\int_{\Omega\times\rb^n}v dd^{\#}u_1\wedge...\wedge dd^{\#}u_m\wedge\beta^{n-m}=-\infty,$
then
$$\lim_{j\rightarrow+\infty}\int_{\Omega\times\rb^n}v dd^{\#}u^{j}_1\wedge...\wedge dd^{\#}u^{j}_m\wedge\beta^{n-m}=-\infty,$$
and our proof is completed.
\end{proof}
Finally, we state our main result in this section, which is the corresponding of Theorem 4.6 in \cite{25} in the setting of $m$-convex functions:
\begin{thm D} Let $u\in\mathscr{C}^{-}_m(\Omega)$, then
$$cap_{m,u}(E)=\int_{E\times\rb^n}(dd^{\#}R_{m,u}(E))^m\wedge\beta^{n-m},$$
for all Borel compact subset $E$ of $\Omega$.
\end{thm D}
We must highlight an important point here, which is that Theorem D is a generalization of Lemma 4.6 in \cite{26}.
\begin{proof} First, assume that $u\in\mathcal{E}^{0}_m(\Omega)$. Since $u\leqslant R_{m,u}(E)\leqslant0$ and in view of implication (\ref{eq1}), we have $R_{m,u}(E)\in\mathcal{E}^{0}_m(\Omega)$. It follows by the definition of $cap_{m,u}$ that $$cap_{m,u}(E)\geqslant\int_{E\times\rb^n}(dd^{\#}R_{m,u}(E))^m\wedge\beta^{n-m}.$$ Conversely, let $v\in\mathscr{C}_m^-(\Omega)\cap L^\infty(\Omega)$ such that $u\leqslant v\leqslant0$. Since  $u\in\mathcal{E}^{0}_m(\Omega)$, then $v\in\mathcal{E}^{0}_m(\Omega)$. Moreover, $u=R_{m,u}(E)$ outside an $m$-polar subset on $E$, so that $R_{m,u}(E)\leqslant v$ outside an $m$-polar subset on $E$. Since functions from $\mathcal{E}_{m}^0(\Omega)$ put no mass on $m$-polar sets, then it follows from Lemma \ref{ll}, the implication (\ref{eq1}) and Theorem \ref{mc} that
$$\begin{array}{lcl}
\ds\int_{E\times\rb^n}(dd^{\#}v)^m\wedge\beta^{n-m}&\leqslant&\ds\int_{E\times\rb^n}(dd^{\#}\max(R_{m,u}(E),v))^m\wedge\beta^{n-m}\\&\leqslant&\ds\int_{\Omega\times\rb^n}(dd^{\#}\max(R_{m,u}(E),v))^m\wedge\beta^{n-m}\\&\leqslant&\ds\int_{\Omega\times\rb^n}(dd^{\#}R_{m,u}(E))^m\wedge\beta^{n-m}\\&=&\ds\int_{E\times\rb^n}(dd^{\#}R_{m,u}(E))^m\wedge\beta^{n-m}.
\end{array}$$
Then, $$cap_{m,u}(E)=\int_{E\times\rb^n}(dd^{\#}R_{m,u}(E))^m\wedge\beta^{n-m},\qquad\forall u\in\mathcal{E}^{0}_m(\Omega).$$
In the general case, if $u\in\mathscr{C}^{-}_m(\Omega)$ there exists a decreasing sequence $(u_j)_j\subset\mathcal{E}_{m}^0(\Omega)$ such that $u_j$ decreases to $u$ as $j\rightarrow+\infty$. By the first step, we have
\begin{equation}\label{eq2}
cap_{m,u_j}(E)=\int_{E\times\rb^n}(dd^{\#}R_{m,u_j}(E))^m\wedge\beta^{n-m},\qquad\forall j.
\end{equation}
On the one hand, Proposition \ref{pr3} (4) implies that $R_{m,u_j}(E)$ decreases to $R_{m,u}(E)$ and thanks to Theorem \ref{l2}, we have
$$\begin{array}{lcl}
\ds\lim_{j\rightarrow+\infty}\int_{E\times\rb^n}(dd^{\#}R_{m,u_j}(E))^m\wedge\beta^{n-m}&=&\ds\lim_{j\rightarrow+\infty}\int_{\Omega\times\rb^n}(dd^{\#}R_{m,u_j}(E))^m\wedge\beta^{n-m}\\&=&\ds\int_{\Omega\times\rb^n}(dd^{\#}R_{m,u}(E))^m\wedge\beta^{n-m}\\&=&\ds\int_{E\times\rb^n}(dd^{\#}R_{m,u}(E))^m\wedge\beta^{n-m}<+\infty.
\end{array}$$
On the other hand, Proposition \ref{pr4} implies that $cap_{m,u_j}(E)$ increases to $cap_{m,u}(E)$. Thus, we obtain the desired result by letting $j\rightarrow+\infty$ in (\ref{eq2}).
\end{proof}
\begin{cor} Let  $u\in\mathscr{C}^{-}_m(\Omega)$, then:
\begin{enumerate}
\item If $E$ is a Borel compact subset of $\Omega$, then $cap_{m,u}(E)<+\infty$.
\item If $u\in\mathcal{F}_m(\Omega)$, then
$$\int_{\Omega\times\rb^n}(dd^{\#}R_{m,u}(E))^m\wedge\beta^{n-m}\leqslant cap_{m,u}(E)\leqslant\int_{\Omega\times\rb^n}(dd^{\#}u)^m\wedge\beta^{n-m},$$
for all open subset $E$ of $\Omega$. In particularly, if $E=\Omega$, then
$$cap_{m,u}(\Omega)=\int_{\Omega\times\rb^n}(dd^{\#}u)^m\wedge\beta^{n-m}.$$
\end{enumerate}
\end{cor}
\begin{proof} (1) Assume that $E$ is a Borel compact subset of $\Omega$. In spite of Proposition \ref{pr3}, we have $R_{m,u}(E)\in\mathcal{F}_m(\Omega)$. Thus, Theorem D and the second statement of Proposition \ref{pr11} imply that $$cap_{m,u}(E)=\int_{E\times\rb^n}(dd^{\#}R_{m,u}(E))^m\wedge\beta^{n-m}<+\infty.$$
(2) Assume that $u\in\mathcal{E}_{m}^0(\Omega)$ and let $v\in\mathscr{C}_m(\Omega)$ such that $u\leqslant v\leqslant0$. In view of Proposition \ref{pr11} we have $\max(u,v)=v\in\mathcal{E}_m^0(\Omega)$, and therefore by (\ref{eq1}) we obtain $$\int_{K\times\rb^n}(dd^{\#}v)^m\wedge\beta^{n-m}\leqslant\int_{\Omega\times\rb^n}(dd^{\#}v)^m\wedge\beta^{n-m}\leqslant\int_{\Omega\times\rb^n}(dd^{\#}u)^m\wedge\beta^{n-m},$$
for every compact $K\subset E$. By taking the supremum over all $v$ and over all $K$, we get
$$cap_{m,u}(E)\leqslant\int_{\Omega\times\rb^n}(dd^{\#}u)^m\wedge\beta^{n-m}.$$
 In order to prove the other inequality, let $(K_j)_j$ be an exhaustive sequence of compact subsets of $E$ ($K_j\subset K_{j+1}$ and $\cup_j K_j=E$). We claim that $R_{m,u}(K_j)$ decreases to $R_{m,u}(E)$. Indeed, it is clear that $R_{m,u}(K_j)$ decreases to $w\geqslant R_{m,u}(E)$, for some $w\in\mathscr{C}^{-}_m(\Omega)$. To show that $R_{m,u}(E)\leqslant w$, we see by definition of the weighted capacity that for each $j$, there exists an $m$-polar subset $F_j\subset K_j$ such that $R_{m,u}(K_j)=u$ on $K_j\smallsetminus F_j$. Hence, for $F=\cup_jF_j$, it is clear that $w=u$ on $E\smallsetminus F$ and therefore $w\leqslant R_{m,u}(E)$. Thus, by Theorem \ref{l2} and Theorem D, we get
$$\begin{array}{lcl}
\ds\int_{\Omega\times\rb^n}(dd^{\#}R_{m,u}(E))^m\wedge\beta^{n-m}&=&\ds\lim_{j\l +\infty}\int_{\Omega\times\rb^n}(dd^{\#}R_{m,u}(K_j))^m\wedge\beta^{n-m}\\&=&\ds\lim_{j\l +\infty}cap_{m,u}(K_j)\leqslant cap_{m,u}(E).
\end{array}$$
Now, consider $(u_j)_j\subset\mathcal{E}_{m}^0(\Omega)$ such that $u_j$ decreases to $u$ and $$\sup_j\int_{\Omega\times\rb^n}(dd^{\#}u_j)^m\wedge\beta^{n-m}<+\infty.$$ By the preceding argument, for each $j$, we have 
$$\int_{\Omega\times\rb^n}(dd^{\#}R_{m,u_j}(E))^m\wedge\beta^{n-m}\leqslant cap_{m,u_j}(E)\leqslant\int_{\Omega\times\rb^n}(dd^{\#}u_j)^m\wedge\beta^{n-m}.$$
In order to complete the proof, it suffices to apply Proposition 11 and Theorem 7, by taking into account the fact that $R_{m,u_j}(E)\in\mathcal{E}_m^0(\Omega)$ decreases to $R_{m,u}(E)\in\mathcal{F}_m(\Omega)$. 
\end{proof}
\section*{Acknowledgment}
The second named author likes to express his gratitude towards Professor Jean-Pierre Demailly for his hospitality and for many stimulating discussions and comments during his visit to "Institut Fourier".

\end{document}